\documentclass[a4paper,reqno,11pt]{amsart}

\usepackage[left=1 in, right=1 in,top=1 in, bottom=1 in]{geometry}

\usepackage{amsfonts,amsmath,amsthm,amssymb,stmaryrd}
\usepackage{mathrsfs}
\usepackage{color}
\allowdisplaybreaks

\newcommand{\beq}{\begin{equation}}
\newcommand{\eeq}{\end{equation}}
\newcommand{\bal}{\begin{align}}
\newcommand{\eal}{\end{align}}
\renewcommand{\(}{\left(}
\renewcommand{\)}{\right)}
\renewcommand{\[}{\left[}
\renewcommand{\]}{\right]}

\providecommand{\se}[1]{\mathcal{E}_{#1}}

\providecommand{\fe}[1]{\mathfrak{E}_{#1}}
\providecommand{\fd}[1]{\mathfrak{D}_{#1}}
\providecommand{\fg}[1]{\mathcal{G}_{#1}}

\providecommand{\seb}[1]{\bar{\mathcal{E}}_{#1}}

\providecommand{\fdb}[1]{\bar{\mathfrak{D}}_{#1}}
\providecommand{\abs}[1]{\left\vert#1\right\vert}
\providecommand{\as}[1]{\abs{#1}^2}
\providecommand{\norm}[1]{\left\Vert#1\right\Vert}
\providecommand{\ns}[1]{\norm{#1}^2}
\providecommand{\Rn}[1]{\mathbb{R}^{#1}}

\def\pa{\partial}
\def\nab{\nabla}
\def\al{\alpha}
\def\dt{\partial_t}
\def\dtt{ \frac{d}{dt}}
\def\hal{\frac{1}{2}}

\def\ls{\lesssim}
\def\p{\partial}
\def\dts{ \frac{d}{ds}}
\def\eps{\varepsilon}
\def\epsd{{\varepsilon,\delta}}
\def\i{\mathcal{I}}
\def\Q{\mathcal{Q}^m}
\def\V{\mathcal{V}^m}
\def\B{\mathcal{B}^m}
\def\dis{\displaystyle}

\def\n{\mathcal{N}}

\def\da{\Delta_{\mathcal{A}}}

\def\daed{\Delta_{\mathcal{A}^\epsd}}
\def\naba{\nab_{\mathcal{A}}}
\def\nabai{\nab_{\mathcal{A}^i}}
\def\nabao{\nab_{\mathcal{A}^1}}

\def\nabaed{\nab_{\mathcal{A}^\epsd}}
\def\diva{\diverge_{\mathcal{A}}}

\def\divao{\diverge_{\mathcal{A}^1}}

\def\divaed{\diverge_{\mathcal{A}^\epsd}}
\def\a{\mathcal{A}}

\def\me{\Lambda_\infty^\eps}
\def\mee{\Lambda_\infty}

\def\mm{\mathcal{M}_0^m}
\def\Mm{\mathcal{M}_0^m}
\def\Mmd{\mathcal{M}_0^{m,\delta}}

\providecommand{\jump}[1]{\left\llbracket #1 \right\rrbracket }

\def\curl{{\rm curl}}

\newcommand{\Div}{\operatorname{div}}
\DeclareMathOperator{\diverge}{div}

\newtheorem{prop}{Proposition}
\newtheorem{thm}[prop]{Theorem}
\newtheorem{lem}[prop]{Lemma}

\newtheorem{rem}[prop]{Remark}

\numberwithin{equation}{section}
\numberwithin{prop}{section}

\title[MHD Contact Discontinuities]{Existence of Multi-dimensional Contact Discontinuities for the Ideal Compressible Magnetohydrodynamics}

\author{Yanjin Wang}
\address{School of Mathematical Sciences\\Xiamen University\\Xiamen, Fujian 361005, China}
\email[Y. J. Wang]{yanjin$\_$wang@xmu.edu.cn}

\author{Zhouping Xin}
\address{The Institute of Mathematical Sciences\\The Chinese University of Hong Kong\\Shatin, NT, Hong Kong}
\email[Z. P. Xin]{zpxin@ims.cuhk.edu.hk}

\thanks{Y. J. Wang was supported by the National Natural Science Foundation of China (12171401, 11771360) and the Natural Science Foundation of Fujian Province of China (2019J02003). Z. P. Xin was supported by Zheng Ge Ru Foundation, HongKong RGC Earmarked Research Grants CUHK14301421, CUHK14300819, CUHK14302819, CUHK14300917, CUHK14302917, and Basic and Applied Basic Research Foundation of Guangdong Province 2020B1515310002.}



\date{\today}

\begin{document}

\begin{abstract}
We establish the local existence and uniqueness of multi-dimensional contact discontinuities for the ideal compressible  magnetohydrodynamics (MHD) in Sobolev spaces, which are most typical interfacial waves for astrophysical plasmas and prototypical fundamental waves for hyperbolic systems of conservation laws. Such waves are characteristic discontinuities for which there is no flow across the discontinuity surface while the magnetic field crosses transversely, which lead to a two-phase free boundary problem where the pressure, velocity and magnetic field are continuous across the interface whereas the entropy and density may have jumps. To overcome the difficulties of possible nonlinear Rayleigh--Taylor instability and loss of derivatives, here we use crucially the Lagrangian formulation and Cauchy's celebrated integral (1815) for the magnetic field. These motivate us to define two special good unknowns; one enables us to capture the boundary regularizing effect of the transversal magnetic field on the flow map, and the other one allows us to get around the troublesome boundary integrals due to the transversality of the magnetic field. In particular, our result removes the additional assumption of the Rayleigh--Taylor sign condition required by Morando, Trakhinin and Trebeschi (\emph{J. Differential Equations} \textbf{258} (2015), no. 7, 2531--2571; \emph{Arch. Ration. Mech. Anal.} \textbf{228} (2018), no. 2, 697--742) and holds for both 2D and 3D and hence gives a complete answer to the  two open questions raised therein. Moreover, there is {\it no loss of derivatives} in our well-posedness theory. The solution is constructed as {\it the inviscid limit} of solutions to some suitably-chosen nonlinear approximate problems for the two-phase compressible viscous non-resistive MHD.
\end{abstract}

\maketitle


\section{Introduction}

\subsection{Eulerian formulation}\label{eulerian}
Consider the equations for the ideal compressible  magnetohydrodynamics (MHD):
\beq\label{MHD}
\begin{cases}
\dt\varrho+\Div (\varrho u)=0
\\
  \partial_t  (\varrho{u} )+  \Div (\varrho {u}\otimes    {u}- B\otimes B) +\nabla (P+\hal |B|^2) = 0\\ \dt B   -  \curl(u\times B) =0
\\ \Div B=0
\\   \partial_t  (\varrho S)+  \Div  ( \varrho u S )=0 ,
\end{cases}
\eeq
where $\varrho$  denotes the density of the plasma, $u$ the velocity, $B$ the magnetic field, $S$ the entropy and $P$ the pressure given by the state equation:
\beq\label{pressure_law}
P= A\varrho^{\gamma} e^{S}\text{ with constants } A>0,\ \gamma>1.
\eeq
The fourth equation in \eqref{MHD} can be transformed to be   an initial constraint, which follows from the third equation. Note that \eqref{MHD} is hyperbolic if $\varrho>0$.

Let $\Sigma(t) $ be a surface of strong discontinuity for \eqref{MHD}  and consider the solutions  that are smooth on either side of $\Sigma(t)$. When there is no flow across the discontinuity surface, the classical Rankine--Hugoniot jump conditions across $\Sigma(t)$ for \eqref{MHD} read as (see  \cite{LLP})
\beq\label{MHDbd1200}
V = u_\pm  \cdot \nu,\  \jump{B}\cdot \nu =0,\  \jump{P +\dis\hal |B|^2}= 0 ,\  B\cdot \nu\jump{u}\cdot \tau =0,\  B\cdot \nu\jump{B}\cdot \tau =0    \text{ on } \Sigma(t),
\eeq
where $V $ denotes the normal velocity of $\Sigma(t)$,  $f_\pm:=f|_{\Omega_\pm(t)}$ for $\Omega_\pm(t)$ the two domains separated by $\Sigma(t)$, $\jump{f}:=f_+-f_-$, $\nu$ is the  unit normal to $\Sigma(t)$ (pointing to $\Omega_-(t)$) and $\tau$ denotes for each tangential vector to $\Sigma(t)$. The first condition in \eqref{MHDbd1200} states that such discontinuities are characteristic discontinuities  for \eqref{MHD}, which are characteristic free boundaries, and the second condition can be also transformed to be an initial constraint.

According to the classification of strong discontinuities in MHD \cite{LLP}, there are two types of characteristic discontinuities  for \eqref{MHD}, $i.e.$, MHD contact discontinuities and MHD tangential discontinuities (current-vortex sheets).
For MHD contact discontinuities, the magnetic field is nowhere tangent to the discontinuity; in such case, \eqref{MHDbd1200} is equivalent to
\beq\label{MHDbd12}
V = u_\pm  \cdot \nu,\quad B_\pm\cdot \nu\neq 0,\quad \jump{P}= 0,\quad \jump{u}  = 0  ,\quad    \jump{B}  = 0    \text{ on } \Sigma(t),
\eeq
Note that in \eqref{MHDbd12} the entropy $S$  and the density $\varrho$  may have jumps.   The boundary conditions \eqref{MHDbd12} are most typical for astrophysical plasmas where  magnetic fields originate in a planet, star, or other rotating object, with a dynamo operating inside, but intersect the surface  \cite{GP}. While for current-vortex sheets, the magnetic field is  tangent to the discontinuity; in such case, \eqref{MHDbd1200} is equivalent to
\beq\label{MHDb11}
V = u_\pm  \cdot \nu,\quad B_\pm\cdot \nu =0,\quad \jump{P +\dis\hal |B|^2}= 0   \text{ on } \Sigma(t).
\eeq
Note that in \eqref{MHDb11} the tangential velocity, tangential magnetic field  and  pressure may also have jumps. Current-vortex sheets usually occur in laboratory plasmas aimed at thermonuclear energy production  \cite{GP}.

 In this paper, we focus on the existence of MHD contact discontinuities for \eqref{MHD}. For simplicity, we will consider \eqref{MHD} in the  horizontally periodic slab $\Omega=\mathbb{T}^2\times (-1,1)$ and   assume that  the interface of discontinuity $\Sigma(t)$ extends to infinity in every horizontal direction and  does not intersect   the upper and lower boundaries $\Sigma_\pm:=\mathbb{T}^2\times \{ \pm1\}$.   $\Sigma_\pm $ are assumed to be    impermeable and perfectly conducting, so
\beq\label{MHDbd71}
u\cdot e_3=0,\quad E \times e_3=0\text{ on }\Sigma_\pm,
 \eeq
  where $E=u\times B$ is the electric field and $e_3=(0,0,1)$.  Note that if $B\cdot e_3\neq 0$ on $\Sigma_\pm$, then \eqref{MHDbd71} is equivalent to (see also \cite{YM})
\beq\label{MHDbd1}
u =0   \text{ on }\Sigma_\pm.
\eeq

To complete the statement of the problem, one must specify the initial data. Suppose that the initial interface of discontinuity $\Sigma(0)$ is given, which does not intersect    $\Sigma_\pm$. This yields the open sets $\Omega_\pm(0)$ on which  the initial density $ \varrho(0)=\varrho_0: \Omega_\pm(0) \rightarrow  \mathbb{R}_+$, the initial velocity $ {u}(0)=u_0: \Omega_\pm(0) \rightarrow  \mathbb{R}^3$, the initial magnetic field  $ B(0)=B_0: \Omega_\pm(0) \rightarrow  \mathbb{R}^3$ and the initial entropy $ S(0)=S_0: \Omega_\pm(0) \rightarrow  \mathbb{R}$ are specified.

\subsection{Related works}

Contact discontinuities (vortex sheets and entropy waves), along with shock waves and rarefaction waves, are fundamental waves in the entropy solutions to multidimensional hyperbolic systems of conservation laws.

Contact discontinuities of the Euler equations are subject to both the Kelvin--Helmholtz instability and the Rayleigh--Taylor instability, see
 Ebin \cite{E2}  for the ill-posedness of incompressible Kelvin--Helmholtz and Rayleigh--Taylor problems and Guo and Tice \cite{GT} for the ill-posedness of the compressible Rayleigh--Taylor problem.
It was shown by the normal mode analysis that  supersonic   vortex sheets in 2D  of  the compressible Euler equations are neutrally linearly stable, while  subsonic vortex sheets in 2D and vortex sheets in 3D are  violently unstable, see Syrovatskij \cite{S2}, Miles \cite{M} and Fejer and Miles \cite{FM}. Coulombel and Secchi \cite{CS1,CS2}  proved the nonlinear stability of 2D supersonic   vortex sheets  for the compressible isentropic Euler equations,  and the nonisentropic case was proved by Morando and Trebeschi \cite{MT} and Morando, Trebeschi and Wang \cite{MTW}. It is known that the surface tension has  the stabilizing effect
on the Kelvin--Helmholtz  and   Rayleigh--Taylor instabilities, see Cheng, Coutand and Shkoller \cite{CCS1} and Shatah and Zeng \cite{SZ1,SZ2} for the local well-posedness of the two-phase incompressible  Euler equations with surface tension and Stevens \cite{S} for the compressible case.

Chen and Wang \cite{CW} and Trakhinin \cite{T1,T} proved the nonlinear stability of MHD tangential discontinuities (current-vortex sheets) under some stability condition when  the  tangential magnetic fields on the two sides of the interface of discontinuity are non-colinear. Recall that one has the Syrovatskij linear stability criterion  for the  incompressible  current-vortex sheets \cite{S1},  and one may refer to Coulombel, Morando, Secchi and Trebeschi \cite{CMST} for the a priori  estimates of the nonlinear problem under a stronger stability condition and Sun, Wang and Zhang \cite{SWZ1} for the local well-posedness under the Syrovatskij  stability condition.  These works show the strong stabilizing effect of tangential magnetic fields on the Kelvin--Helmholtz instability.

It should be noted that these known stability results depend also crucially on the fact that there is an elliptic equation for the front function of the discontinuity surface  in both 2D vortex sheets and 3D current-vortex sheets with the non-colinear tangential magnetic fields, and the Rayleigh--Taylor instability is also absent in \cite{CS2,MTW,CW,T,SWZ1}.
Although MHD contact discontinuities are neutrally linearly stable, see Blokhin and Trakhinin \cite{BT}, and do not admit the Kelvin--Helmholtz instability, yet their front symbols are not elliptic, and they allow the possibility of the Rayleigh--Taylor instability due to the nonlinear effect. Thus whether the magnetic field can prevent such a nonlinear instability is a subtle and difficult issue. Indeed, in contrast to the analysis of current-vortex sheets, even the a priori nonlinear tangential estimates  pose essential difficulties due to the regularity issue of the interface of discontinuity. To overcome such difficulties, Morando, Trakhinin and Trebeschi \cite{MTT1,MTT} proposed the Rayleigh--Taylor sign condition on the jump of the normal derivative of the pressure across the interface of discontinuity, and proved the existence of 2D MHD contact discontinuities under this additional Rayleigh--Taylor sign condition; however, they also pointed out that their ideas and approaches, developed in \cite{MTT1,MTT}, fail to treat some key boundary integral terms appearing in 3D. Therefore, the general 3D case and the  question  whether  the Rayleigh--Taylor sign condition is necessary for the existence of MHD contact discontinuities in both 2D and 3D were left as open problems in \cite{MTT1,MTT}. In this paper, by  exploring the fact that  the  magnetic fields for MHD contact discontinuities are always transversal across the interface of  discontinuity and thus may have stabilizing effects and some key observations related to the boundary integrals, we resolve these two open questions by establishing the existence of both 2D and 3D MHD contact discontinuities in Sobolev spaces without any additional condition. Thus our results show that the Rayleigh--Taylor sign condition is not necessary for  the existence of MHD contact discontinuities in both 2D and 3D. We refer to Trakhinin and Wang \cite{TW} for a recent work of the nonlinear stability of the two-phase ideal compressible MHD with the surface tension introduced in \eqref{MHDbd12}.

\subsection{Lagrangian reformulation}\label{lagrangian}
Differently from \cite{MTT1,MTT}, our analysis in this paper relies crucially on the reformulation of the problem under consideration in Lagrangian coordinates.
Take $\Omega_\pm:=\{x_3\gtrless 0\}\cap\Omega$ as the Lagrangian domains and  denote the interface by $\Sigma:=\{x_3=0\}$. Assume that there is a  diffeomorphism $\eta_0:\Omega_\pm\rightarrow \Omega_\pm(0)$ such that
\begin{equation} \label{eta0in}
\jump{\eta_0} =\jump{\p_3\eta_0} =0\text{ on } \Sigma,\ \Sigma(0)=\eta_0(\Sigma) \text{ and }\Sigma_\pm=\eta_0(\Sigma_\pm).
\end{equation}
Note that \eqref{eta0in} is fulfilled at least when the initial interface $\Sigma(0)$ is a graph as in \cite{MTT1,MTT}; indeed, if $\Sigma(0)$ is given by the graph of a function $x_3=h_0(x_1,x_2)$, then $\eta_0$ can be constructed simply as $\eta_0:=(x_1,x_2,x_3+\chi(x_3)h_0(x_1,x_2))$ for some   function $\chi(x_3)$ with $\chi(0)=1$ and $\chi(\pm1)=0$. Define the flow map  $\eta(t,x)\in\Omega_\pm(t)$ by that for $x\in \Omega_\pm$,
\begin{equation}\label{etaequ}
\begin{cases}
\partial_t\eta(t,x)=u(t,\eta(t,x)),\  t>0
\\ \eta(0,x)=\eta_0(x).
\end{cases}
\end{equation}

Assume that $\eta(t,\cdot)$ is invertible and define the Lagrangian unknowns in $\Omega_\pm$:
\begin{equation}
 (\rho,v, b,s,p)(t,x):= (\varrho,u, B,S,P)(t,\eta(t,x)).
\end{equation}
Throughout the rest of the paper, an equation on $\Omega$ means that the equation holds in both $\Omega_+$ and $\Omega_-$. In the Lagrangian coordinates  $\dt s=0$, which implies $s =s_0:=S_0(\eta_0) $, and the problem  under consideration can be reformulated equivalently as
\begin{equation}\label{MHDv0}
\begin{cases}
\partial_t\eta =v &\text{in } \Omega\\
\frac{1}{\gamma p}\dt p+  \diva v = 0 &\text{in  }\Omega\\
\rho\partial_tv  +\naba  (p+\hal|b|^2) =b\cdot\naba b  &\text{in } \Omega\\
\partial_tb +b\diva v=b\cdot\naba v &\text{in } \Omega\\
\diva b =0 &\text{in  }\Omega\\
\jump{p}=0,\quad\jump{v}=0,\quad \jump{b}=0 \quad  &\text{on }\Sigma\\
v=0 \quad  &\text{on }\Sigma_\pm\\
 (\eta,p,v, b)\mid_{t=0} =(\eta_0, p_0,v_0, b_0),
 \end{cases}
\end{equation}
where, by \eqref{pressure_law},
\beq\label{hha1}
\rho=\rho_0p_0^{-\frac{1}{\gamma}}p^\frac{1}{\gamma}\text{ with }\rho_0:=A^{-\frac{1}{\gamma}}e^{-\frac{s_0}{\gamma}  }p_0^{\frac{1}{\gamma}}.
\eeq
Here  $\mathcal A:=(\nabla\eta)^{-T}$,  $(\naba)_i=\p_i^\a:=\a_{ij}\p_j $, $\diva:=\naba\cdot$ and $\jump{f}:=f_+-f_-$ for $f_\pm:=f|_{\Omega_\pm} .$ Denote also   $J := {\rm det}(\nabla\eta)$, the Jacobian of the coordinate transformation, and
\beq\label{nndef}
\n:=J\a e_3=\p_1\eta\times\p_2\eta.
\eeq
Recall that \eqref{eta0in} and the following initial conditions have been required:
\beq\label{eta0in2}
\rho_0,p_0>0,  J_0\neq0\text{ and } \diverge_{\a_0}b_0=0\text{ in }\Omega,\ \jump{ b_0}\cdot \n_0=0\text{ on }\Sigma,\  b_0\cdot \n_0\neq0\text{ on }\Sigma\cup\Sigma_\pm,
\eeq
where $J_0=J(\eta_0)$, $\a_0=\a(\eta_0)$ and $\n_0=\n(\eta_0)$.
Some additional conditions for the initial data will be specified later. Recall that $\rho_0$ (and $s_0$)  may have a jump and can be regarded as a parameter function.

\subsection{Expressions of $\rho,p$ and $b$}\label{expsec}

It is crucial to deduce from \eqref{MHDv0} some important facts. First, by the first equation in \eqref{MHDv0}, one has
\beq \label{jequ}
\partial_t J=J\diverge_\a \dt\eta =J\diverge_\a v .
\eeq
Thus  \eqref{jequ}, the second equation in \eqref{MHDv0} and \eqref{hha1}  yield $\partial_t ( \rho J) =0$ and hence
\beq \label{jequ1}
\rho =\rho_0J_0J^{-1}\text{ and }p=p_0 J_0^\gamma J^{-\gamma},
\eeq
while \eqref{jequ} and the fourth and first equations  in \eqref{MHDv0}  imply
\beq
\partial_t (Jb) =J b\cdot\naba v \equiv J \nabla v \a^T b
=J \nabla \dt\eta \a^T b =- J \nabla  \eta \dt\a^T b,
\end{equation}
which shows $\partial_t (J\a^T b)=0$ and hence
\begin{equation}\label{re00}
b=J^{-1}J_0\a_0^Tb_0 \cdot \nabla \eta.
\end{equation}
We may refer to \eqref{re00} as Cauchy's integral  for the magnetic field in Lagrangian coordinates as its analogue to Cauchy's celebrated integral for the vorticity of the compressible Euler equations \cite{C}.

As a consequence, one has the following facts for \eqref{MHDv0}.

\begin{prop}\label{prop1}
$(i)$   $\dt(J\diva b) =0$; $(ii) $  $\dt(b\cdot \n)=0$.
\end{prop}
\begin{proof}
By \eqref{re00}, one has that, using the Piola identity $\p_j (J\a_{ij})=0$,
\beq
\dt(J\diva b )=  \dt\diverge (J\a^Tb)=  0
\eeq
and that, recalling \eqref{nndef},
\beq
\dt(b\cdot \n)= \dt(J\a^T b)_3= 0.
\eeq
The proposition is thus concluded.
\end{proof}

\begin{prop}\label{prop2}
Assume that  $ \jump{\eta_0} =\jump{\p_3\eta_0} = \jump{b_0}=0$, $\jump{p_0}  =0$ and $b_0\cdot\n_0\neq0$ on $\Sigma$. Then
\beq\label{ggdd}
 \jump{\p_3v}=\jump{\eta}=\jump{\p_3\eta}=0 \text{ on }\Sigma.
\eeq
\end{prop}
\begin{proof}
Since $\partial_t\eta =v$, by   \eqref{re00} and \eqref{jequ1}, one obtains
\begin{align}\label{hha2}
& b_0\cdot \n_0 \p_3 v \equiv J_0\a_0^Tb_0 \cdot  \nabla v  -(J_0\a_0^Tb_0)_\beta    \p_\beta v
 \nonumber\\&\quad=\dt(Jb)-(J_0\a_0^Tb_0)_\beta    \p_\beta v =J_0p_0^{\frac{1}{\gamma}}\dt(p^{-\frac{1}{\gamma}}b)-(J_0\a_0^Tb_0)_\beta \p_\beta v.
\end{align}
Then by $ \jump{\eta_0} =\jump{\p_3\eta_0} = \jump{b_0}=0$ and $\jump{p_0}  =0$  on $\Sigma$ and the jump conditions in \eqref{MHDv0}, one has
\beq\label{hha222}
 b_0\cdot \n_0 \jump{\p_3 v } =J_0p_0^{\frac{1}{\gamma}}\jump{\dt(p^{-\frac{1}{\gamma}}b)}-(J_0\a_0^Tb_0)_\beta \jump{\p_\beta v}=0\text{ on }\Sigma,
\eeq
which implies
$\jump{\p_3v}=0$ on $\Sigma $ since $ b_0\cdot \n_0\neq 0$ on $\Sigma$.
Since $\dt\jump{\eta} =\jump{v}=0$ and $\dt\jump{\p_3\eta}=\jump{\p_3v}=0$ on $\Sigma$, $\jump{\eta}=\jump{\p_3\eta}=0$ on $\Sigma$ thus follows again from $ \jump{\eta_0} =\jump{\p_3\eta_0} =0$ on $\Sigma$.
\end{proof}

\section{Main results}

 \subsection{Statement of the results}
 We will work in a high-regularity context with regularity up to $m$ temporal derivatives, which requires one to use the data $(\eta_0,p_0,v_0,b_0,\rho_0)$  of \eqref{MHDv0} to construct the initial data  $ (\partial_t^j p(0),\partial_t^j v(0),\partial_t^j b(0))$ for $j=1,\dotsc,m$ (and $\dt^j\eta(0)=\dt^{j-1}v(0)$ for $j=1,\dots,m-1$) recursively by using \eqref{in0def}. These data need to satisfy various conditions \eqref{compatibility}, which require $(\eta_0,p_0,v_0,b_0,\rho_0)$ to satisfy the necessary $(m-1)$-th order compatibility conditions that are natural for the local well-posedness of \eqref{MHDv0} in the functional framework below.

Let $H^k(\Omega_\pm)$, $k\ge 0$ and $H^s(\Sigma )$, $s \in \Rn{}$ be the usual Sobolev spaces with norms denoted by $\norm{\cdot}_k  $ and $ \abs{\cdot}_s$, respectively. For $f=f_\pm$ in $\Omega_\pm$, denote $\ns{f}_k := \ns{f_+}_{H^k(\Omega_+)} + \ns{f_-}_{H^k(\Omega_-)} $. For an integer $m\ge 0$,   define the  high-order energy as
\beq\label{eedef}
\se{m}:=\sum_{j=0 }^m \norm{(\dt^j p,\dt^j v,\dt^j b)}_{m-j }^2 +\norm{\eta}_m^2 +\abs{\eta}_m^2 .
\eeq
Denote
\beq\label{mdef1}
\mm:=P\(\norm{(\eta_0,p_0,v_0,b_0,\rho_0)}_{m}^2 + \abs{\eta_0}_m^2\),
\eeq
where $P$ is a generic polynomial.

Our main result of this paper is stated as follows.
\begin{thm}\label{the_main}
Let  $m\ge 4$ be an integer. Assume that   $\eta_0\in H^m(\Omega_\pm)\cap H^m(\Sigma)$ and $p_0,v_0, b_0,\rho_0\in H^m(\Omega_\pm)$ are given such that $\diverge_{\a_0}b_0=0$ in $\Omega$,
\beq\label{iidd100}
  \jump{\eta_0} =\jump{\p_3\eta_0} =0\text{ and }\jump{ b_0}\cdot \n_0=0  \text{ on } \Sigma,\  \eta_{0,3}=\pm 1\text{ on }\Sigma_\pm,
\eeq
\beq\label{iidd1}
\rho_0, p_0, |J_0|\ge c_0>0
\text{ in }\Omega\text{ and } |b_0\cdot \n_0|\ge c_0>0\text{ on }\Sigma\cup\Sigma_\pm \text{ for some constant }c_0>0
\eeq
and  the $(m-1)$-th order compatibility conditions \eqref{compatibility} are satisfied.
Then there exist a $T_0>0$ and a unique solution $(\eta,p,v,b)$ to  \eqref{MHDv0} on the time interval $[0, T_0]$ which satisfies
\begin{equation}\label{enesti}
\sup_{t \in [0,T_0]} \se{m}(t) \leq \Mm .
\end{equation}
\end{thm}

\begin{rem}
Our result in particular removes the assumption of the Rayleigh--Taylor sign condition required in \cite{MTT1,MTT} and holds for both 2D and 3D and hence gives a complete answer to the two open questions raised therein. This shows also the strong stabilizing effect of the transversal magnetic field on  the Rayleigh--Taylor instability.  The key ingredient here is  the new boundary regularity $\abs{\eta}_m^2$, which is captured from the regularizing effect of the transversal  magnetic field.
\end{rem}

\begin{rem} Note that there is no loss of derivatives in our well-posedness theory in Sobolev spaces, which is in contrast to  \cite{MTT1,MTT,TW} where the solution is constructed by employing the Nash--Moser-type linearized iteration scheme and thus has a loss of derivatives.
\end{rem}

\begin{rem}
Theorem \ref{the_main} holds also for the cases $\Omega=\mathbb{R}^2\times(-1,1)$ or $\Omega=\mathbb{R}^3$ provided that we replace  $(\eta,p,v,b,\rho)$ in the definitions \eqref{eedef} and \eqref{mdef1} by $(\eta-Id,p-\bar p,v-\bar v,b-\bar b,\rho-\bar \rho)$, etc., where $(\bar p,\bar v,\bar b,\bar \rho)$ is a trivial contact-discontinuity state with the discontinuity surface $\{x_3=0\}$ satisfying $\bar p,\bar\rho>0$ and $\bar b_3\neq 0$.
\end{rem}

Our solution to \eqref{MHDv0} is constructed as the inviscid limit of solutions to ``well-chosen" nonlinear viscous approximate problems. For this, we need to first smooth the data $(\eta_0, p_0,v_0, b_0,\rho_0) $ given in Theorem \ref{the_main} to produce the regular enough data $(\eta_0^\delta, p_0^\delta,v_0^\delta, b_0^\delta,\rho_0^\delta) $ for \eqref{MHDv0}, with
the smoothing parameters $\delta>0$, which satisfies all the initial conditions (with $c_0$ replaced by $c_0/2$) assumed in Theorem \ref{the_main} except that $\diverge_{\a_0^\delta}b_0^\delta=0$ in $\Omega$ may not hold, where $\a_0^\delta=\a(\eta_0^\delta)$; such construction will be elaborated in Appendix \ref{datasec}. Now we consider the following viscous (and non-resistive) approximate problem: for the artificial viscosity $\eps>0$,
\begin{equation}\label{MHDve}
\begin{cases}
\partial_t\eta^\epsd =v^\epsd &\text{in } \Omega\\
\frac{1}{\gamma p^\epsd}\dt p^\epsd+  \divaed v^\epsd = 0 &\text{in  }\Omega\\
\rho^\epsd\partial_tv^\epsd  +\nabaed  (p^\epsd+\hal|b^\epsd|^2)-\varepsilon\daed v^\epsd =b^\epsd\cdot\nabaed b^\epsd+ \Psi^\epsd  &\text{in } \Omega\\
\partial_tb^\epsd +b^\epsd\divaed v^\epsd=b^\epsd\cdot\nabaed v^\epsd &\text{in } \Omega\\
\jump{p^\epsd}=0,\quad\jump{v^\epsd}=0,\quad \jump{b^\epsd}=0,\quad \jump{\p_3v^\epsd}=0 \quad  &\text{on }\Sigma\\
v^\epsd=0 \quad  &\text{on }\Sigma_\pm\\
 (\eta^\epsd,p^\epsd,v^\epsd, b^\epsd )\mid_{t=0} =(\eta_0^\delta, p_0^\delta,v_0^\delta, b_0^\delta )
 \end{cases}
\end{equation}
with
\beq\label{ppeeq}
\rho^\epsd= {\rho_0^\delta}(p_0^\delta)^{-\frac{1}{\gamma}}(p^\epsd)^{\frac{1}{\gamma}} ,
\eeq
where $\a^\epsd=\a(\eta^\epsd)$ and $\daed=\divaed\nabaed$.
   Note that all the arguments in Section \ref{expsec} hold also for \eqref{MHDve}; in particular, according to $(i)$ in Proposition \ref{prop1}, one has
   \beq\label{dividen}
   J^\epsd\divaed b^\epsd =       J_0^\delta\diverge_{\a_0^\delta} b_0^\delta\text{ in }\Omega,
   \eeq
   where $J^\epsd=J(\eta^\epsd)$ and $J_0^\delta=J(\eta_0^\delta)$. It should be pointed out that it is important to introduce the correctors $\Psi^\epsd$ in \eqref{MHDve}, defined according to \eqref{MHDvk211002233223} and \eqref{MHDvk2110022332232}, that vanish  as $\eps\rightarrow0$, so that the smoothed data $(\eta_0^\delta, p_0^\delta,v_0^\delta, b_0^\delta,\rho_0^\delta) $ satisfies the $(m-1)$-th order compatibility conditions \eqref{compatibility0je}  and thus can be taken as the data for \eqref{MHDve}. It is crucial that the boundary conditions of \eqref{MHDve} are essentially same as those of  \eqref{MHDv0} (cf. Proposition \ref{prop2}), however, the jump conditions on $\Sigma$  are not standard for solving the viscous MHD (it seems that \eqref{MHDve} would be over-determined!) and so it is not direct to get the local well-posedness of \eqref{MHDve}, even with  $\eps>0$. Our way of getting around this difficulty is to follow first those of the compressible Navier--Stokes equations (see for instance \cite{DS} for the references) to get the local well-posedness of the modified problem \eqref{MHDve0ee}, $i.e.$, the corresponding problem with the jump conditions   in   \eqref{MHDve}  replaced by the following  ``standard"  jump   conditions
\beq\label{mdfb}
 \jump{v^\epsd} =0,\quad\jump{   \nabaed v^\epsd  }\n^\epsd =0 \text{ on }\Sigma,
 \eeq
where $\n^\epsd =\n(\eta^\epsd)$. The crucial point  is then that under the initial conditions these two sets of jump conditions are indeed equivalent. The full details will be provided in Section \ref{appsec}, and the  local well-posedness of \eqref{MHDve} will be recorded in Theorem \ref{lwp}.

In order to pass to the limit as $\epsd\rightarrow 0$ in \eqref{MHDve}, one needs to show that the solution $(\eta^\epsd, p^\epsd,v^\epsd, b^\epsd)$ constructed in Theorem \ref{lwp} actually exists on an $(\epsd)$-independent time interval and satisfies certain uniform estimates. Set
\beq
 Z_{3}= x_3(x_3^2-1)\partial_{3}\text{ and }Z^\alpha = \p_{1}^{\alpha_{1}} \p_{2}^{\alpha_{2}} Z_{3}^{\alpha_{3}} \text{ for }\alpha\in\mathbb{N}^3.
\eeq
 Define the anisotropic Sobolev norm
\beq\label{ani_sob_norm}
 \norm{  f  }_{k,\ell} :=\sum_{\alpha\in\mathbb{N}^3,|\alpha| \leq \ell}\norm{ Z^\alpha f }_{k}.
\eeq
Note that $\norm{  f  }_{k}=\norm{  f  }_{k,0}$.
  For an integer $m\ge 4$ and each $\varepsilon>0$,  define the  energy functional
  \beq\label{fgdef}
  \fg{m}^\eps(t):=
  \sup_{[0,t]}  \fe{m}^{\varepsilon}+\int_0^{t} \(\fdb{m}^{\varepsilon}+\fd{m}^{\varepsilon} \),
  \eeq
  where
  \begin{align}\label{edef}
&\fe{m}^{\varepsilon}:=\sum_{j=0 }^{m} \norm{(\dt^j  p,\dt^j  v,\dt^j  b)}_{0,m-j }^2  +\sum_{j=0 }^{m-1} \norm{\dt^j  v}_{1,m-j-1 }^2+\sum_{j=0 }^{m-1} \norm{\dt^j \p_3 b\cdot\n}_{0,m-j-1 }^2  \nonumber\\& \qquad\quad+\abs{\eta}_m^2+   \ns{\eta}_m
+ \eps\ns{\eta}_{1,m}+ \eps^2\ns{\eta}_{m+1},
\\\label{dbbdef}
&\fdb{m}^{\varepsilon}:= \varepsilon   \sum_{j=0 }^{m}\norm{\dt^j   v}_{1,m-j }^2 ,
\\\label{ddef}
 &\fd{m}^{\varepsilon}:=  \sum_{j=0 }^{m-1} \norm{(\dt^j  p,\dt^j  b)}_{m-j }^2   +\sum_{j=0 }^{m-2}  \norm{\dt^j  v}_{m-j }^2  +\varepsilon^2 \sum_{j=0 }^{m-1}\norm{\dt^j  v}_{m-j +1}^2.
\end{align}

\begin{thm}\label{the_main2}
Let  $m\ge 4$ be an integer. Let $ (\eta_0^\delta, p_0^\delta , v_0^\delta,b_0^\delta,\rho_0^\delta)$ be the smoothed data constructed in Appendix \ref{datasec} and $\Phi^\epsd$ be the corrector   defined according to \eqref{MHDvk211002233223} and \eqref{MHDvk2110022332232}.  Then there exist a $T_0>0$ and a $\delta_0>0$ such that for each $0<\delta<\delta_0$ there exists an $\eps_0=\eps_0(\delta)>0$ so that for $0<\eps\le \eps_0$  the unique solution  $(\eta^\epsd, p^\epsd,v^\epsd, b^\epsd)$ to  \eqref{MHDve}, constructed in Theorem \ref{lwp}, exists on $[0,T_0]$ and  satisfies
\begin{equation}\label{enesti2}
 \fg{m}^\eps(\eta^\epsd, p^\epsd,v^\epsd, b^\epsd)(T_0) \leq \Mm .
\end{equation}
\end{thm}

By Theorem \ref{the_main2},  one can easily pass to the limit as first $\varepsilon\rightarrow 0$ and then $\delta\rightarrow 0$ in \eqref{MHDve} to find that  the limit $(\eta,p,v, b)$ of   $(\eta^\epsd, p^\epsd,v^\epsd, b^\epsd)$ solves  \eqref{MHDv0} on  $[0, T_0]$, where the constraint that $\diva b=0$ in $\Omega$ is recovered by using $(i)$ in Proposition \ref{prop1} as $\diverge_{\a_0}b_0=0$ in $\Omega$. Moreover, the solution satisfies the estimate $ \fg{m}^0(T_0)\le \Mm $. By the a posterior estimates, one can improve those $L^2$-in-time estimates in $ \fg{m}^0(T_0)$ to be  $L^\infty$-in-time so that the estimate \eqref{enesti} holds. After proving the uniqueness of the solutions, Theorem \ref{the_main} is thus concluded.

\begin{rem}
Our analysis can be applied to justify the inviscid limit of the compressible viscous non-resistive  MHD in  bounded domains with the no-slip boundary condition in Sobolev spaces when the magnetic field is nowhere tangent to   the boundary.
It should be pointed out that unlike the Navier--Stokes equations, there are no boundary layers in the inviscid limit here. This is due to that here the  viscous and ideal problems have the exactly same boundary conditions.
\end{rem}

 \subsection{Strategy of the uniform estimates}\label{sec2.3}

 The main part of the paper will be devoted to prove Theorem  \ref{the_main2}, where the  key step  is to derive the  uniform estimate  \eqref{enesti2} on a  time interval small but independent of $\epsd$. Suppress the dependence of solutions on $\epsd$. Note that the estimates of the lower order derivatives of $(p,v,b)$ and  the last three terms of $\eta$ in \eqref{edef} can be easily controlled by using the transported estimates as recorded in Section \ref{lasts6}. It then suffices to estimate the highest order derivatives of $(p,v,b)$ and the boundary regularity of $\eta$.

To derive the highest order tangential energy estimates, we shall use the  equations \eqref{MHDv0s} (derived from \eqref{MHDve}) for $(q,v,b)$ with  $q=p+\hal |b|^2$   the total pressure. One starts with applying the tangential {\it spatial}  derivatives $Z^m$, any $Z^\al$ for $\al\in \mathbb{N}^3$ with $|\al|=m$,  to  \eqref{MHDv0s}. Typically,   the estimate of the commutator between $Z^m$ and $\pa_i^\a$ needs a control of $\norm{Z^m \nabla\eta}_0$, which yields a loss of one derivative. Motivated by Alinhac \cite{A}, a natural  way to get around this difficulty is  to introduce the good unknowns,
\beq\label{gggd}
\mathcal{V}^m= Z^m v -
Z^m \eta\cdot\naba v,\ \mathcal{Q}^m= Z^m q -
Z^m \eta\cdot\naba q\text{ and }\mathcal{B}^m= Z^m b -
Z^m \eta\cdot\naba b.
\eeq
Then  the highest order term of $\eta$ will be canceled when considering the equations satisfied by good unknowns. This  leads to that,   by using $\dt\eta=v $ and $(ii)$ in Proposition \ref{prop1},
\begin{align} \label{s02}
& \dfrac{1}{2}\dfrac{d}{dt}\int_{\Omega} J\( \frac{1}{\gamma p} \abs{\Q -b\cdot \B}^2 +\rho \abs{\mathcal{V}^m }^2 + \abs{\mathcal{B}^m }^2 \)  +\eps \int_{\Omega}J \abs{ \naba \mathcal{V}^m}^2\nonumber
 \\&= \int_{\Sigma }\jump{ \mathcal{Q}^m}  \mathcal{V}^m\cdot\n- \int_\Sigma  b\cdot\n \jump{\mathcal{B}^m} \cdot \mathcal{V}^m+{\sum}_{\mathcal{R}}\nonumber
\\& =  -\int_{\Sigma } J^{-1} Z^m   \eta\cdot\n \jump{\pa_3 q}  \dt Z^m \eta \cdot\n+\int_{\Sigma }  b_0^\delta\cdot\n_0^\delta J^{-1} Z^m   \eta\cdot\n \jump{\pa_3 b} \cdot  \dt Z^m \eta+{\sum}_{\mathcal{R}}.
\end{align}
Here $\n_0^\delta=\n(\eta_0^\delta)$ and ${\sum}_{\mathcal{R}}$ denotes terms, whose time integration, after some delicate
arguments, can be bounded by $\Mm+\eps\Mmd  +  t^{1/2}P(\fg{m}^\eps(t) )$ for $0<\delta<\delta_0$ with some $\delta_0>0$, where
\beq\label{mdef12}
\Mmd :=   P_\delta\(\ns{(\eta_0,p_0,v_0,b_0,\rho_0)}_{m} \).
\eeq
It is well known that the geometric symmetry structure for the first term in the right hand side of \eqref{s02} is crucial:
\begin{align} \label{s022}
  -\int_{\Sigma } J^{-1} Z^m   \eta\cdot\n \jump{\pa_3 q}  \dt Z^m \eta \cdot\n =  -\hal\dtt\int_{\Sigma } \jump{\pa_3 q} J^{-1} \abs{Z^m   \eta\cdot\n  }^2 +{\sum}_{\mathcal{R}},
\end{align}
and this would yield the boundary regularity $\abs{Z^m \eta \cdot\n}_0^2$ if one assumed the Rayleigh--Taylor sign condition ($i.e.$, $\jump{\pa_3 q}>0$ on $\Sigma$).
However, there is no such symmetry for the second term in the right hand side of \eqref{s02}. Note that this term vanishes when $ b_0^\delta\cdot\n_0^\delta=0$ on $\Sigma$, and it seems out of control when $ b_0^\delta\cdot\n_0^\delta\neq0$ on $\Sigma$. Our way to overcome this difficulty is to make use of Cauchy's integral  \eqref{re00}, which implies, by \eqref{jequ1},
\beq\label{s0202}
b\cdot\naba  = J^{-1}J_0^\delta (\a_0^\delta)^Tb_0^\delta   \cdot\nabla= \rho(\rho_0^\delta)^{-1} b_0^\delta \cdot \nabla_{\a_0^\delta} .
\eeq
We shall use \eqref{s0202} for the term $b\cdot\nabla_\a b$ in the second equation of \eqref{MHDv0s}, which
  allows one to  introduce instead in \eqref{gggd},
\beq\label{good2}
\mathcal{B}^m= Z^m b - Z^m \eta_0^\delta\cdot\nabla_{\a_0^\delta} b.
\eeq
Due to \eqref{good2}, the second term in the right hand side of \eqref{s02} is changed  to be
\begin{align}   \label{s03}
&\int_{\Sigma } b_0^\delta\cdot\n_0^\delta (J_0^\delta)^{-1} Z^m   \eta_0^\delta \cdot\n_0^\delta \jump{\pa_3 b} \cdot \dt Z^m \eta\nonumber
\\&\quad=\dtt\int_{\Sigma } b_0^\delta\cdot\n_0^\delta (J_0^\delta)^{-1} Z^m   \eta_0^\delta \cdot\n_0^\delta \jump{\pa_3 b} \cdot  Z^m  \eta+{\sum}_{\mathcal{R}}.
\end{align}
By \eqref{s022} and \eqref{s03}, one can then deduce from \eqref{s02}, correspondingly, that
\begin{align}    \label{s04}
  \norm{  ( p,   v,   b)(t) }_{0,m}^2  +\eps\int_0^t   \norm{   v }_{1,m }^2
   \ls \as{Z^m   \eta(t)}_0+  \Mm+\eps\Mmd    +  t^{1/2}P(\fg{m}^\eps(t) )  .
\end{align}
Now to control $\as{Z^m\eta }_0$ in the right hand side of \eqref{s04}, our key point here is to use further Cauchy's integral  \eqref{re00} in $\norm{  b }_{0,m}^2$ and then introduce the good unknown
\beq \label{s044}
\Xi^m:=Z^m \eta-  Z^m \eta_0^\delta \cdot\nabla_{\a_0^\delta} \eta.
\eeq
These allow one to  add  $\norm{(\a_0^\delta)^Tb_0^\delta \cdot \nabla \Xi^m}_0^2$ to the left hand side of \eqref{s04}. Then the boundary regularizing effect of the magnetic field due to that $((\a_0^\delta)^Tb_0^\delta)_3=(J_0^\delta)^{-1}b_0^\delta\cdot\n_0^\delta \neq0$ near $\Sigma$ is   captured by applying Lemma \ref{tra_es} to $\Xi^m $:
\beq\label{s05}
  \abs{\Xi^m }_0^2 \ls \norm{(\a_0^\delta)^Tb_0^\delta \cdot \nabla \Xi^m }_0\norm{\Xi^m }_0+  \ns{\Xi^m }_0 .
\eeq
By \eqref{s05} and Cauchy's inequality, one can then improve \eqref{s04} to be
\begin{align}   \label{s055}
  \norm{  ( p,   v,   b)(t) }_{0,m}^2+ \as{    \eta(t)}_m  +\eps\int_0^t   \norm{   v }_{1,m }^2
   \le    \Mm+\eps\Mmd   +  t^{1/2}P(\fg{m}^\eps(t) ) .
\end{align}
Similarly but in a much simpler way,  the rest of highest order tangential energy estimates involving at least one time derivative can be also controlled by the same bound as \eqref{s055}.

Now we turn to the derivation of  the normal derivatives estimates near the boundary $\Sigma\cup\Sigma_\pm$. For the original ideal MHD \eqref{MHDv0}, utilizing again the transversality of the magnetic field near the boundary as in Yanagisawa \cite{Y} and Yanagisawa and Matsumura \cite{YM}, the estimates of normal derivatives of the solution can be derived by expressing them in terms of tangential derivatives. However, for the viscous approximation \eqref{MHDve}, one needs to explore additionally the ODE-in-time structures and certain cancelation related to the viscous term. More precisely, first, by \eqref{dividen} and the fourth and second equations in \eqref{MHDve}, one can express $\p_3 b\cdot \n$ and $\p_3 v$ as sums of tangential derivatives of $p,v,b$, denoted by ${\sum}_{\mathcal{T}}$, up to a multiplication of  continuous functions of $\nabla\eta, p,v,b$. Next, one uses the tangential part of the third equation in \eqref{MHDve} to deduce that
\begin{align}\label{ode1}
& \p_3   b \cdot \tau^\beta  +\eps \frac{ \a_{k3} \a_{k3}}  {(\a^Tb)_3^2}  \p_3 \( b\cdot\naba v\)\cdot\tau^\beta={\sum}_{\mathcal{T}}+\eps \nabla \p_\beta  v+{\sum}_{\eps} ,\ \beta=1,2.
\end{align}
where $ \tau^\beta,\beta=1,2,$ are defined in \eqref{base} and ${\sum}_{\eps}$ denotes the terms that can be ultimately controlled by $ \eps\fg{m}^\eps$.
By the fourth and second equations in \eqref{MHDve}, $b\cdot\naba v=\dt b\underline{- \frac{b}{\gamma p}  \dt p}$ and hence \eqref{ode1} can be regarded as an ODE (in time) for  $\p_3b\cdot\tau^\beta$. On the other hand, the normal part of the third equation in \eqref{MHDve} yields that
 \begin{align}\label{ode2}
 \p_3 p\underline{+ \p_3  b \cdot \tau^\beta  b \cdot \tau^\beta } -\varepsilon\p_3 \diva v  ={\sum}_{\mathcal{T}}+\eps \nabla \p_\beta  v+{\sum}_{\eps} .
\end{align}
By the   second equations in \eqref{MHDve}, $\diva v=-\frac{1}{\gamma p} \dt p$ and hence \eqref{ode2} is an ODE  for $\p_3p$.
Therefore, basing on these structures above, by a recursive  argument in terms of the numbers of normal derivatives,
one can then deduce the desired estimates so that
\beq\label{s078}
\fg{m}^\eps(t)\ls \Mm+\eps \Mmd   +  t^{1/2}P(\fg{m}^\eps(t) )+ \eps\fg{m}^\eps(t).
\eeq
It should be noted that it is crucial in deriving \eqref{s078} that \eqref{ode1} and \eqref{ode2} have an exact cancelation  between the two underlined crossing terms; otherwise, they seem out of control.
One thus concludes the estimate \eqref{enesti2} from \eqref{s078} for $0<\eps\le \eps_0(\delta)$ and some $ T_0>0.$

 \subsection{Notation}
The Einstein convention of summing over repeated certain indices will be used, and the repeated over the latin letter $i,j,k,\ell$ is from $1$ to $3$ while  the repeated over the greek letter $ \beta,\gamma$ is from $1$ to $2$.
Throughout the paper $C$ denotes for generic  positive constants and $P$ for generic   polynomials that do
not depend on  $\epsd$, which are allowed to change from line to line. $C_\delta,P_\delta$, etc. denote  the additional dependence. $A_1 \lesssim A_2$ means that $A_1 \le C A_2$.

 $\mathbb{N} = \{ 0,1,2,\dotsc\}$ denotes for the collection of non-negative integers. When using space-time differential multi-indices, we write $\mathbb{N}^{1+d} = \{ \alpha = (\alpha_0,\alpha_1,\dotsc,\alpha_d) \}$ to emphasize that the $0-$index term is related to temporal derivatives. For just spatial multi-indices, we write $\mathbb{N}^d$.  For $\alpha \in \mathbb{N}^{1+2}$,  $\pa^\alpha = \dt^{\alpha_0}  \pa_1^{\alpha_1} \pa_2^{\alpha_2},$  for $\alpha \in \mathbb{N}^{3}$,  $Z^\alpha =  \pa_1^{\alpha_1} \pa_2^{\alpha_2}Z_3^{\al_3}$ and  for $\alpha \in \mathbb{N}^{1+3}$,  $\bar Z^\alpha = \dt^{\alpha_0} \pa_1^{\alpha_1} \pa_2^{\alpha_2}Z_3^{\al_3}.$ Denote the standard commutator
\beq  \label{cconm1}
\[\pa^\al, f\]g=\pa^\al (fg)-f\pa^\al g
\eeq
and the symmetric commutator
\beq  \label{cconm2}
\[\pa^\al, f,g \]=\pa^\al (fg)-f\pa^\al g-\pa^\al f g.
\eeq

We omit the differential elements $dx$ and $dx_1dx_2$ of the integrals over $\Omega_\pm$ and $\Sigma$   and also sometimes the  differential elements $ds$ of the time integrals.

\section{Nonlinear viscous approximation}\label{appsec}

\subsection{Initial data and compatibility conditions of \eqref{MHDv0}}\label{sec71}

For the   data  $(\eta_0,p_0,v_0,b_0,\rho_0)$ of \eqref{MHDv0} given in Theorem \ref{the_main}, one needs to construct the initial data $(\partial_t^j p(0),\partial_t^j v(0),\partial_t^j  b(0))$ for $j=1,\dotsc,m$ and $\dt^j\eta(0)$ for $j=1,\dots,m-1$  recursively by using
\begin{equation}\label{in0def}
\begin{pmatrix}
\partial_t^j p(0)\\\partial_t^j v(0)\\\partial_t^j  b(0)
\end{pmatrix}:= \partial_t^{j-1}
\begin{pmatrix}
-\gamma p  \diva v   \\
 \rho^{-1}\(b\cdot\naba b-\naba  (p+\hal|b|^2)\)  \\
b\cdot\naba v-b\diva v
\end{pmatrix}{\Big|}_{t=0},\ j=1,\dots,m
\end{equation}
and
\begin{equation}\label{in0def2}
 \dt^j\eta(0):=\dt^{j-1}v(0),\ j=1,\dots,m-1,
\end{equation}
where, as before, $\rho={\rho_0 } p_0^{-\frac{1}{\gamma}}p^{\frac{1}{\gamma}}$ and $\a=\a(\eta)$. By the iteration, \eqref{in0def} and \eqref{in0def2} enable one to determine these initial data in terms of $(\eta_0,p_0,v_0,b_0,\rho_0)$ and its spatial derivatives in such a way which is essentially same as that one determines the time derivatives of the solution $(\eta,p,v,b)$ to \eqref{MHDv0} in terms of $(\eta,p,v,b,\rho)$ by using \eqref{MHDv0} repeatedly. Moveover, due to the first initial condition in \eqref{iidd1},
it is straightforward to check that
\beq
\se{m}(0)\le \Mm .
\eeq

  In order for  $(\eta_0,p_0,v_0,b_0,\rho_0)$ to be taken as the   data for the local well-posedness of  \eqref{MHDv0} in our energy functional framework,  these data  need to satisfy, besides \eqref{iidd100} and \eqref{iidd1} (and $\diverge_{\a_0}b_0=0$ in $\Omega$), the following  $(m-1)$-th order  compatibility conditions:
\beq\label{compatibility}
 \jump{\dt^j p(0)}  =0\text{ and}   \jump{\dt^j v(0)} = \Pi_0\jump{\dt^j b(0)} =0\text{ on }\Sigma \text{ and } \dt^j v(0)  =0\text{ on }\Sigma_\pm,\  j=0,\dots,m-1,
\eeq	
where $\Pi_0=I-|\n_0|^{-2}\n_0\otimes\n_0$.
\begin{lem}\label{lemcom}
Under the initial conditions \eqref{iidd100}, \eqref{iidd1} and \eqref{compatibility}, it holds that
\beq \label{compatibility11}
 \jump{\dt^j b(0)}\cdot\n_0 =0\text{ on }\Sigma ,\  j=0,\dots,m-1
\eeq	
and
\beq\label{compatibility22}
 \jump{\p_3\dt^j v(0)} =0\text{ on }\Sigma ,\  j=0,\dots,m-2.
\eeq
\end{lem}
\begin{proof}
We prove \eqref{compatibility11} by the induction. First, \eqref{compatibility11} holds for $j=0$ due to the last condition on $\Sigma$ in \eqref{iidd100}. Now suppose that \eqref{compatibility11} holds for $j=0,\dots,\ell$ with $\ell\in [0,m-2].$ Note that by the definitions \eqref{in0def} and \eqref{in0def2}, similar to the remarks below \eqref{in0def2},  $\dt^{\ell+1}(b\cdot\n)(0)$ is determined in terms of $(\eta_0,p_0,v_0,b_0,\rho_0)$ in the same way as the one that $\dt^{\ell+1}(b \cdot \n) $ is expressed in terms of $(\eta,p,v,b,\rho)$ by using \eqref{MHDv0} repeatedly. According to $(ii)$ in Proposition \ref{prop1}, one thus has
\beq\label{goodd1}
\dt^{\ell+1}(b\cdot\n)(0)=0.
\eeq
 Then by \eqref{goodd1}, the last two jump conditions in \eqref{compatibility} and the induction assumption, one deduces that, recalling the commutator notation \eqref{cconm1},
\beq
 \jump{\dt^{\ell+1} b(0)}\cdot\n_0 =-\jump{\[\dt^{\ell+1},\n\] b(0)}=0\text{ on }\Sigma.
\eeq	
This concludes \eqref{compatibility11}.

We now prove \eqref{compatibility22}. By the definitions \eqref{in0def} and \eqref{in0def2} again and according to \eqref{hha2}, similarly as for \eqref{goodd1}, one has
\beq\label{hha2''}
 b_0\cdot \n_0 \p_3 \dt^j v (0)  =J_0p_0^{\frac{1}{\gamma}}\dt^{j+1}(p^{-\frac{1}{\gamma}}b)(0) -(J_0\a_0^Tb_0)_\beta \p_\beta \dt^j v(0) ,\ j\ge 0.
\eeq
Then by the jump conditions in \eqref{iidd100},   \eqref{compatibility} and \eqref{compatibility11}, one obtains that for $j=0,\dots,m-2,$
\beq
b_0\cdot\n_0\jump{\p_3\dt^j v(0)} =J_0p_0^{\frac{1}{\gamma}} \jump{\dt^{j+1}(p^{-\frac{1}{\gamma}}b)(0)}
-(J_0\a_0^Tb_0)_\beta  \jump{ \p_\beta \dt^j v(0) }=0\text{ on }\Sigma,
\eeq
which implies \eqref{compatibility22} since $b_0\cdot\n_0\neq 0$ on $\Sigma.$
\end{proof}

\subsection{Smoothed initial data and correctors for \eqref{MHDve} }\label{sec72}

We will construct solutions to \eqref{MHDv0} as the inviscid limit of the corresponding problem for the viscous non-resistive MHD.  However, such an approximation scheme is highly technical due to the issue of the boundary conditions and the high order compatibility conditions for the initial data. Our idea here is to regularize  \eqref{MHDv0} by the viscous approximate problem \eqref{MHDve}, where the smoothed data $(\eta_0^\delta, p_0^\delta,v_0^\delta, b_0^\delta,\rho_0^\delta) $ and the so-called corrector $\Psi^\epsd$ are introduced.

Let $(\eta_0,p_0,v_0,b_0,\rho_0)$ be the   data of   \eqref{MHDv0} given in Theorem \ref{the_main} and $(\eta_0^\delta, p_0^\delta,v_0^\delta, b_0^\delta,\rho_0^\delta) $ be the smoothed data constructed in Appendix \ref{datasec}, with the smoothing parameter  $\delta>0$. Let $(\partial_t^j p^\delta(0),\partial_t^j v^\delta(0),\partial_t^j  b^\delta(0))$ for $j=1,\dotsc,m$ and  $\dt^j\eta^\delta(0)$ for $j=1,\dots,m-1$ be constructed recursively by using \eqref{in0def} and \eqref{in0def2}, with the data $(\eta_0,p_0,v_0,b_0,\rho_0)$ replaced by $(\eta_0^\delta, p_0^\delta,v_0^\delta, b_0^\delta,\rho_0^\delta) $, $\rho$ replaced by $\rho^\delta={\rho_0^\delta }( p_0^\delta)^{-\frac{1}{\gamma}}(p^\delta)^{\frac{1}{\gamma}}$ and $\a$ replaced by $\a^\delta=\a(\eta^\delta)$. Then it follows from the construction in Appendix \ref{datasec} and Lemma \ref{lemcom} (for the smoothed data correspondingly) that
\begin{equation} \label{eta0inde}
\jump{\eta_0^\delta} =\jump{\p_3\eta_0^\delta} =0\text{ on } \Sigma \text{ and } \eta_{0,3}^\delta=\pm 1\text{ on }\Sigma_\pm
\end{equation}
and
\begin{equation} \label{eta0inde2}
\begin{split}
&\jump{\partial_t^j p^\delta(0)} =0\text{ and }\jump{\partial_t^j v^\delta(0)} =\jump{\partial_t^j b^\delta(0)} =0\text{ on } \Sigma\text{ and } \partial_t^j v^\delta(0)=0\text{ on }\Sigma_\pm,
\\&\ j=0,\dots,m-1 \text{ and }\jump{\p_3\partial_t^j v^\delta(0)} = 0\text{ on } \Sigma,\ j=0,\dots,m-2.
\end{split}
\end{equation}
Moreover,
 \beq
 \eta_0^\delta\rightarrow \eta_0 \text{ in }H^m(\Omega_\pm)\cap H^m(\Sigma)\text{ and } (   p_0^\delta, v_0^\delta,  b_0^\delta,  \rho_0^\delta)\rightarrow (p_0,v_0,b_0,\rho_0) \text{ in }H^m(\Omega_\pm) \text{ as }\delta\rightarrow0,
 \eeq
and for $0<\delta<\delta_0$ with some $\delta_0>0$ (hereafter),
\beq\label{fadl1}
\rho_0^\delta, p_0^\delta, |J_0^\delta|\ge \frac{c_0}{2}>0
\text{ in }\Omega \text{ and } |b_0^\delta\cdot \n_0^\delta|\ge \frac{c_0}{2}>0\text{ on }\Sigma\cup\Sigma_\pm
\eeq
and
\beq\label{gesa1}
\se{m}(\eta^\delta,p^\delta,v^\delta,b^\delta)(0)\le \Mm \text{ and }\se{l}(\eta^\delta,p^\delta,v^\delta,b^\delta)(0)\le \Mmd ,\ l\ge m+1.
\eeq

Next, define the corrector  $\Psi^\epsd$ such that
\beq\label{MHDvk211002233223}
\dt^j ((\rho^\delta)^{-1} \Psi^\epsd)(0)=-\dt^j((\rho^\delta)^{-1} \eps\Delta_{\a^\delta}v^\delta)(0),\  j=0,\dots,m-3\text{ and } m-1
\eeq
and
\beq\label{MHDvk2110022332232}
\dt^{m-2} ((\rho^\delta)^{-1} \Psi^\epsd)(0)=-\dt^{m-2}((\rho^\delta)^{-1} \eps\Delta_{\a^\delta}v^\delta)(0)+ (\dt^{m-1}v^\delta (0))^\eps-\dt^{m-1}v^\delta (0),
\eeq
where $(\dt^{m-1}v^\delta (0))^\eps:=\phi_\eps\ast\mathcal{E}_\Omega(\dt^{m-1}v^\delta (0))$ for $\phi_\eps$ the standard mollifier in $\mathbb{R}^3$ and $\mathcal{E}_\Omega$ the Sobolev extension operator. The existence of such $\Psi^\epsd$ is standard (see \cite{LM}). Note that the classical properties of mollifiers (see \cite{A}) imply
\beq\label{gesa2}
\norm{(\dt^{m-1}v^\delta (0))^\eps-\dt^{m-1}v^\delta (0)}_k\ls \eps^l  \norm{ \dt^{m-1}v^\delta (0)}_{k+l},\ \forall k,l\ge 0.
\eeq
By the definition of $\Psi^\epsd$, one deduces from \eqref{gesa1}, \eqref{fadl1} and \eqref{gesa2} that
\begin{align}\label{psies}
\sup_{[0,\infty)}\sum_{j=0}^m\ns{\dt^j\Psi^\epsd}_{m-j}&\ls \eps^2 \Mmd+\ns{(\dt^{m-1}v^\delta (0))^\eps-\dt^{m-1}v^\delta (0)}_m
\nonumber\\&\ls \eps^2 \Mmd+\eps^2  \norm{ \dt^{m-1}v^\delta (0)}_{m+1}^2\le  \eps^2 \Mmd .
\end{align}

Now take $(\eta_0^\delta, p_0^\delta,v_0^\delta, b_0^\delta,\rho_0^\delta) $ and $\Psi^\epsd$ in the above as the ones for the viscous approximate problem \eqref{MHDve}. One then constructs the corresponding data $(\partial_t^j p^\epsd(0),\partial_t^j v^\epsd(0),\partial_t^j  b^\epsd(0))$ for $j=1,\dotsc,m$ and $ \partial_t^j  \eta^\epsd(0)$ for $j=1,\dotsc,m-1$ recursively by
\begin{equation}\label{in0defepsd}
\begin{pmatrix}
\partial_t^j p^\epsd(0)\medskip\\\partial_t^j v^\epsd(0)\medskip\\\partial_t^j  b^\epsd(0)
\end{pmatrix}:= \partial_t^{j-1}
\begin{pmatrix}
-\gamma p^\epsd  \divaed v^\epsd\smallskip   \\
 (\rho^\epsd)^{-1}\left(b^\epsd\cdot\nabaed b^\epsd-\nabaed  (p^\epsd+\hal|b^\epsd|^2)\right.
 \\ \hspace{-35pt}\left.+\varepsilon\daed v^\epsd  + \Psi^\epsd \)
 \smallskip \\
b^\epsd\cdot\nabaed v^\epsd-b^\epsd\divaed v^\epsd
\end{pmatrix}{\Big|}_{t=0},\ j=1,\dots,m
\end{equation}
and
\begin{equation}\label{in0defepsd2}
 \dt^j\eta^\epsd(0):=\dt^{j-1}v^\epsd(0),\ j=1,\dots,m-1,
\end{equation}
where $\rho^\epsd={\rho_0^\delta }( p_0^\delta)^{-\frac{1}{\gamma}}(p^\epsd)^{\frac{1}{\gamma}}$ and $\a^\epsd=\a(\eta^\epsd)$.
By comparing \eqref{in0def}--\eqref{in0def2} (with superscript $\delta$ added) and \eqref{in0defepsd}--\eqref{in0defepsd2}, due to \eqref{MHDvk211002233223} and \eqref{MHDvk2110022332232}, one has
\beq\label{tm10}
\begin{split}
&(\partial_t^j p^\epsd(0),\partial_t^j b^\epsd(0))=(\partial_t^j p^\delta(0),\partial_t^j b^\delta(0)),\ j=0,\dots,m-1,
\\
& \partial_t^j v^\epsd(0) = \partial_t^j v^\delta(0) ,\ j=0,\dots,m-2,\  \partial_t^{m-1} v^\epsd(0) = (\partial_t^{m-1} v^\delta(0))^\eps
 \end{split}
 \eeq
and
\beq\label{tm1}
\begin{split}
& \partial_t^m p^\epsd(0) = \partial_t^m p^\delta(0)-\gamma p_0^\delta \diverge_{\a_0^\delta}\((\dt^{m-1}v^\delta (0))^\eps-\dt^{m-1}v^\delta (0)\),
 \\&\partial_t^m v^\epsd(0) = \partial_t^m v^\delta(0)+\eps   (\rho_0^\delta)^{-1} \Delta_{\a_0^\delta}\((\dt^{m-1}v^\delta (0))^\eps-\dt^{m-1}v^\delta (0)\),
  \\&\partial_t^m b^\epsd(0) = \partial_t^m b^\delta(0)+ b_0^\delta\cdot  \nabla_{\a_0^\delta}\((\dt^{m-1}v^\delta (0))^\eps-\dt^{m-1}v^\delta (0)\)
  \\&\qquad\qquad\quad- b_0^\delta   \diverge_{\a_0^\delta}\((\dt^{m-1}v^\delta (0))^\eps-\dt^{m-1}v^\delta (0)\).
 \end{split}
\eeq
Then  by \eqref{tm10}, \eqref{tm1}, \eqref{gesa1} and \eqref{gesa2}, one obtains
\begin{align}\label{gesa122}
 \se{m}(\eta^\epsd,p^\epsd,v^\epsd,b^\epsd)(0)
 & \le \Mm +\Mm \ns{(\dt^{m-1}v^\delta (0))^\eps-\dt^{m-1}v^\delta (0)}_1
 \nonumber
 \\&\quad+\Mm\eps^2 \ns{(\dt^{m-1}v^\delta (0))^\eps-\dt^{m-1}v^\delta (0)}_2\nonumber
 \\& \le \Mm +\Mm \eps^2 \ns{\dt^{m-1}v^\delta (0)}_2   \le \Mm + \eps^2 \Mmd .
\end{align}
Moreover, \eqref{tm10} and \eqref{eta0inde2} imply in particular the following $(m-1)$-th order compatibility conditions for \eqref{MHDve}:
\begin{equation} \label{compatibility0je}
\begin{split}
&\jump{\partial_t^j p^\epsd(0)} =0\text{ and }\jump{\partial_t^j v^\epsd(0)} =\jump{\p_3\partial_t^j v^\epsd(0)}=\jump{\partial_t^j b^\epsd(0)} =0\text{ on } \Sigma
\\&\text{ and } \partial_t^j v^\epsd(0)=0\text{ on }\Sigma_\pm,\ j=0,\dots,m-1.
\end{split}
\end{equation}
It is worth noting that one more compatibility condition other than those in \eqref{eta0inde2}, that is, $\jump{\p_3\partial_t^{m-1} v^\epsd(0)}=0$ on $\Sigma$, is included in \eqref{compatibility0je}, which is necessary for the local well-posedness of the viscous approximation \eqref{MHDve} in the functional framework below in the next subsection; this is exactly the reason why the last two terms in \eqref{MHDvk2110022332232} have been added.

\subsection{Local well-posedness of \eqref{MHDve}}\label{sec73}


Now we turn to the local well-posedness of \eqref{MHDve}.
For an integer $m\ge 3$,   define the  energy functionals
\begin{equation}\label{p_energy_def}
 \mathbb{E}_{2m} : = \sum_{j=0}^{m}   \ns{(\dt^j p ,\dt^j v ,\dt^j b )}_{2m-2j} +\ns{\eta }_{2m+1}
\end{equation}
and
\begin{equation}\label{p_dissipation_def}
 \mathbb{D}_{2m}  :=  \sum_{j=0}^{m+1} \ns{\dt^j v }_{2m-2j+1},
\end{equation}
where $\norm{\cdot}_{-1}$ denotes the norm of $(H_0^1(\Omega))^\ast$.
The existence and uniqueness of solutions to \eqref{MHDve} can be stated as follows.

\begin{thm}\label{lwp}
Let $m\ge 3$ be an integer, $ (\eta_0^\delta, p_0^\delta , v_0^\delta,b_0^\delta,\rho_0^\delta)$ be the smoothed data constructed in Appendix \ref{datasec} and $\Phi^\epsd$ be the corrector  defined according to \eqref{MHDvk211002233223} and \eqref{MHDvk2110022332232}. For each $\epsd>0$, there exist  a $T_0^\epsd >0$ and a unique solution $ (\eta^\epsd,p^\epsd,v^\epsd,b^\epsd)$ to \eqref{MHDve}  on   $[0,T_0^\epsd]$ satisfying
\begin{equation}\label{zero_1}
 \sup_{[0,T_0^\epsd]} \mathbb{E}_{2m}(\eta^\epsd,p^\epsd,v^\epsd,b^\epsd) + \int_0^{T_0^\epsd} \mathbb{D}_{2m}(\eta^\epsd,p^\epsd,v^\epsd,b^\epsd)\le   P_\epsd(\mm).
\end{equation}
Moreover,  it holds that for $t\in[0,T_0^\epsd]$,
\beq\label{zero_123}
\rho^\epsd, p^\epsd, |J^\epsd|\ge \frac{c_0}{4}>0
\text{ in }\Omega\text{ and }|b^\epsd\cdot \n^\epsd|\ge \frac{c_0}{4}>0\text{ on }\Sigma\cup\Sigma_\pm,
\eeq
   \beq\label{dividen22}
   J^\epsd\divaed b^\epsd =       J_0^\delta\diverge_{\a_0^\delta} b_0^\delta\text{ in }\Omega,
   \eeq
   and
\beq
\jump{\eta^\epsd}=\jump{\p_3\eta^\epsd}  =0 \text{ on }\Sigma.
\eeq

\end{thm}
\begin{proof}
Consider first the following modified problem of  \eqref{MHDve}:
\begin{equation}\label{MHDve0ee}
\begin{cases}
\partial_t\eta^\epsd =v^\epsd &\text{in } \Omega\\
\frac{1}{\gamma p^\epsd}\dt p^\epsd+  \divaed v^\epsd = 0 &\text{in  }\Omega\\
\rho^\epsd\partial_tv^\epsd  +\nabaed  (p^\epsd+\hal|b^\epsd|^2)-\varepsilon\daed v^\epsd =b^\epsd\cdot\nabaed b^\epsd+ \Psi^\epsd  &\text{in } \Omega\\
\partial_tb^\epsd +b^\epsd\divaed v^\epsd=b^\epsd\cdot\nabaed v^\epsd &\text{in } \Omega\\
\jump{v^\epsd } =0,\quad\jump{   \nabaed v^\epsd  }\n^\epsd =0\quad  &\text{on }\Sigma\\
v^\epsd=0 \quad  &\text{on }\Sigma_\pm\\
 (\eta^\epsd,p^\epsd,v^\epsd, b^\epsd )\mid_{t=0} =(\eta_0^\delta, p_0^\delta,v_0^\delta, b_0^\delta ).
 \end{cases}
\end{equation}
Since \eqref{jequ1} and \eqref{re00} hold also for the solution to \eqref{MHDve0ee}, one can   eliminate $p^\epsd$ (and $\rho^\epsd$) and $b^\epsd$ from \eqref{MHDve0ee}  to reformulate it equivalently as
\begin{equation}\label{MHDv1e}
\begin{cases}
\partial_t\eta^\epsd =v^\epsd &\text{in } \Omega\\
\rho_0^\delta J_0^\delta (J^\epsd)^{-1}\partial_tv^\epsd  -\varepsilon\daed v^\epsd=((J^\epsd)^{-1}J_0^\delta(\a_0^\delta)^Tb_0^\delta \cdot \nabla)^2\eta^\epsd+ \Psi^\epsd
\\\qquad\qquad\qquad
-\nabaed \(p_0^\delta(  J_0^\delta)^\gamma  (J^\epsd)^{-\gamma} +\hal |(J^\epsd)^{-1}J_0^\delta(\a_0^\delta)^Tb_0^\delta \cdot \nabla \eta^\epsd|^2\)    &\text{in } \Omega\\
 \jump{v^\epsd } =0,\quad\jump{   \nabaed v^\epsd  }\n^\epsd =0&\text{on }\Sigma
\\ v^\epsd =0 &\text{on } \Sigma_\pm
\\
 (\eta^\epsd,v^\epsd)\mid_{t=0} =(\eta_0^\delta, v_0^\delta).
 \end{cases}
\end{equation}
For any fixed $\epsd>0$, similarly as   the compressible Navier--Stokes equations, the right hand side of the second equation in \eqref{MHDv1e} can be easily controlled by the viscosity term within a local time interval, and so the local solution $  (\eta^\epsd ,  v^\epsd   )$ to \eqref{MHDv1e}, in the functional of $\mathbb{E}_{2m}+\int_0^t\mathbb{D}_{2m}$, on   $[0,T_0^\epsd]$ for some $T_0^\epsd>0$   can be constructed by using the same scheme; the conditions \eqref{compatibility0je} imply in particular that the corresponding $(m-1)$-th order compatibility conditions for \eqref{MHDv1e}  hold. We shall omit the details and refer to, for instance, \cite{DS} for the references. Then defining $p^\epsd$ (and $\rho^\epsd$) and $b^\epsd$ by \eqref{jequ1} and \eqref{re00}, respectively, one sees that  $  (\eta^\epsd ,  p^\epsd,v^\epsd,b^\epsd   )$ solves \eqref{MHDve0ee} on   $[0,T_0^\epsd]$ and
satisfies the estimate \eqref{zero_1}. The estimate \eqref{zero_123} follows from  the fundamental theorem of calculus, \eqref{fadl1} and \eqref{zero_1}, by restricting  $T_0^\epsd$   smaller if necessary, while the identity \eqref{dividen22} follows by using $(i)$ in Proposition \ref{prop1}.

Now to conclude the theorem,  it remains  to prove the following jump   conditions:
\beq\label{bdd2}
\jump{p^\epsd }= 0\text{ and }\jump{b^\epsd }= \jump{\p_3 v^\epsd  }  =\jump{\eta^\epsd}=\jump{\p_3\eta^\epsd}=0 \text{ on }\Sigma.
\eeq
Recall that $\{\p_1\eta^\epsd,\, \p_2\eta^\epsd,\,\n^\epsd\}$ is a basis of $\mathbb{R}^3$. By the first jump  condition in   \eqref{MHDve0ee}, one has that $\jump{\eta^\epsd}=0$ on $\Sigma$ by the first equation in \eqref{MHDve0ee} since $\jump{\eta_0^\delta}=0$ on $\Sigma$ and that
\beq\label{jj20}
\jump{\nabaed v^\epsd}\p_\beta\eta^\epsd= \jump{(\p_\beta\eta^\epsd)^T \a^\epsd (\nabla v^\epsd)^T  }
= \jump{  \p_\beta v^\epsd}=0\text{ on }\Sigma,\ \beta=1,2.
\eeq
This together with the second jump  condition in   \eqref{MHDve0ee} implies
\beq\label{jj2}
 \jump{\nabaed v^\epsd}=0 \text{ on }\Sigma .
\eeq
Now,  by the fourth equation in \eqref{MHDve0ee} and \eqref{jj2}, one obtains
\beq\label{jj244}
\partial_t \jump{b^\epsd} +\jump{b^\epsd}\divaed  v^\epsd =\jump{b^\epsd}\cdot\nabaed v^\epsd\text{ on }\Sigma,
\eeq
which implies $\jump{b^\epsd}=0$ on $\Sigma$ since $\jump{b_0^\delta}=0$ on $\Sigma$. Similarly, one has $\jump{p^\epsd}=0$ on $\Sigma$ since $\jump{p_0^\delta}=0$ on $\Sigma$.
One thus concludes \eqref{bdd2} similarly as in the proof of Proposition \ref{prop2}.
\end{proof}

\section{Transported estimates} \label{lasts6}
Now we turn to derive the uniform-in-$(\epsd)$ estimates for the solution  $ (\eta^\epsd,p^\epsd,v^\epsd,b^\epsd)$ to \eqref{MHDve} on $[0,T_0^\epsd]$ constructed in Theorem \ref{lwp}. It should be noted that the first condition in \eqref{zero_123} will be always used without mentioning explicitly.  For notational simplification, we will  keep only the $\eps$-dependence of the functionals such as $\mathcal{G}_m^\eps$,  $\mathfrak{E}_{m}^{\varepsilon }$, $\mathfrak{D}_{m}^{\varepsilon }$, $\fdb{m}^{\varepsilon}$ and the $\delta$-dependence on the data $ (\eta_0^\delta, p_0^\delta , v_0^\delta,b_0^\delta,\rho_0^\delta)$, $\a_0^\delta,$ $J_0^\delta,$ $\n_0^\delta,$ but suppress the dependence of the solution on $\epsd$. We may assume that $T_0^\epsd\le 1$ and restrict $t\in [0,T_0^\epsd]$ in the following.

We begin with the transported estimates. Let $m\ge 4$.
Define
  \begin{align}\label{fdef}
\mathfrak{F}_{m}^{\varepsilon}:=\sum_{j=0 }^{m-1} \norm{(\dt^j  p,\dt^j  v,\dt^j  b)}_{m-j-1 }^2    +   \ns{\eta}_m+ \eps\ns{\eta}_{1,m}+ \eps^2\ns{\eta}_{m+1} .
\end{align}
\begin{prop}\label{tres1}
For $t\in [0,T_0^\epsd]$ with $T_0^\epsd\le 1$, it holds that
\begin{equation}\label{etaest}
\mathfrak{F}_{m}^{\varepsilon}(t)\ls \Mm+\eps \Mmd   +t \fg{m}^{\varepsilon}(t).
\end{equation}
\end{prop}
\begin{proof}
It follows directly from  the fundamental theorem of calculus,
the definitions  of   $\fg{m}^\eps$ and $\mathfrak{F}_{m}^{\varepsilon}$ and using $\dt\eta=v$ that
\beq
\mathfrak{F}_{m}^{\varepsilon}(t)\leq \mathfrak{F}_{m}^{\varepsilon}(0)+t \fg{m}^{\varepsilon}(t).
\eeq
Note that, by \eqref{gesa122} and \eqref{gesa1},
\beq
\mathfrak{F}_{m}^{\varepsilon}(0)\ls \Mm +\eps\Mmd +\eps \ns{\eta_0^\delta}_{m+1}\le \Mm+\eps\Mmd  .
\eeq
Then \eqref{etaest} follows.
\end{proof}

Denote
\beq
 \me(t):=P\Big(\sup_{[0,t]}\mathfrak{F}_{m}^{\varepsilon} \Big).
\eeq
Then by \eqref{etaest},
\beq\label{lalaes}
  \me(t)\le \Mm+\eps\Mmd  +t P(\fg{m}^{\varepsilon}(t)).
\eeq

\section{Tangential energy estimates} \label{lasts5}

In this section, we will derive the  energy evolution estimates for the highest order tangential derivatives of $(p,v,b)$ and the boundary regularity of $ {\eta}$ on $\Sigma$. It is more convenient to use  the following  equations derived from \eqref{MHDve}:
 \begin{equation}\label{MHDv0s}
\begin{cases}
\frac{1}{\gamma p}\dt q-\frac{1}{\gamma p}   b \cdot \dt  b + \diva v = 0 &\text{in  }\Omega\\
\rho\partial_tv  +\naba  q-b\cdot\naba b-\eps \da v=  \Psi^\epsd    &\text{in } \Omega\\
\partial_tb-\frac{b}{\gamma p} \dt q+\frac{b}{\gamma p}  b\cdot \dt b-b\cdot\naba v=0&\text{in } \Omega ,
 \end{cases}
\end{equation}
where $ q=p+\hal |b|^2$ is the total pressure.
 Recall   the boundary conditions:
 \begin{equation}\label{MHDv0sb}
\jump{q}=0\text{ and }\jump{b}=\jump{v}=\jump{\p_3v}=  \jump{\eta}=\jump{\p_3\eta}=   0    \text{ on }\Sigma,\
v=0\text{ and }\eta_3=\pm 1    \text{ on }\Sigma_\pm.
\end{equation}

One considers first the estimates of the highest order tangential spatial derivatives. Let $Z^m$ be any $  Z^\al$ for  $\al\in \mathbb{R}^3$ with $|\al|=m$. To commute $Z^m$ with each term in \eqref{MHDv0s}, it is  useful to establish the following general expressions. Recall the commutator notations \eqref{cconm1} and \eqref{cconm2}.
 For $i=1,2,3,$ one has
\begin{equation}
 Z^m (\pa^\a_if) =  \pa^\a_i Z^m f +\a_{i3}\[Z^m, \pa_3  \]f+ Z^m{\a_{ij}} \pa_j f+\left[Z^m, {\a_{ij}} ,\pa_j f\right]
\end{equation}
and   $Z\a_{ij}=-\a_{i\ell}Z\pa_\ell  \eta_k  \a_{kj}$ implies
\begin{align}
 Z^m\a_{ij} \pa_j f & =-\a_{i\ell}\pa_\ell Z^{m}\eta_k\a_{kj}\pa_j f-\a_{i3}\[Z^{m},\p_3\] \eta_k  \a_{kj}\pa_j f
-\left[Z^{m-1}, \a_{i\ell}\a_{kj} \right]Z \pa_\ell  \eta_k \pa_j f\nonumber
\\& =-\pa^\a_i( Z^{m}\eta\cdot\nabla_\a  f)+Z^{m}\eta\cdot\nabla_\a ( \pa^\a_i f)-\a_{i3}\[Z^{m},\p_3\] \eta_k  \a_{kj}\pa_j f\nonumber
\\&
\quad
-\left[Z^{m-1}, \a_{i\ell}\a_{kj}\right]Z\pa_\ell \eta_k\pa_j f.
\end{align}
It then holds that
\begin{equation}\label{commf}
 Z^m (\pa^\a_if) =  \pa^\a_i\left(Z^m f- Z^{m}\eta\cdot\nabla_\a  f\right)+ \mathcal{C}_i^m(f),
\end{equation}
where
\begin{align}\label{commfno}
\mathcal{C}_i^m(f)=&\a_{i3}\[Z^m, \pa_3  \]f+Z^{m}\eta\cdot\nabla_\a ( \pa^\a_i f)-\a_{i3}\[Z^{m},\p_3\] \eta_k  \a_{kj}\pa_j f
\nonumber
\\&-\left[Z^{m-1}, \a_{i\ell}\a_{kj}\right]Z\pa_\ell \eta_k\pa_j f
+\left[Z^m, {\a_{ij}} ,\pa_j f\right].
\end{align}
It was first observed by Alinhac \cite{A} that the highest order term of $\eta$ will be cancelled when one uses the good unknown $Z^m f- Z^{m}\eta\cdot\nabla_\a  f$, which allows one to perform high order energy estimates.

\begin{lem}
It holds that
\begin{equation}\label{comest}
\norm{\mathcal C^m (f)}_0\leq  P(\norm{\eta}_{m})\norm{f}_m.
\end{equation}
\end{lem}
\begin{proof}
First,  using Sobolev's embedding theorem, one has that since $m\ge 4$,
\begin{align}
&\norm{\a_{i3}\[Z^m, \pa_3  \]f}_0+\norm{Z^{m}\eta\cdot\nabla_\a ( \pa^\a_i f)}_0+\norm{\a_{i3}\[Z^{m},\p_3\] \eta_k  \a_{kj}\pa_j f}_0\nonumber
\\&\quad\le\norm{\a_{i3}  }_{L^\infty} \norm{\[Z^{m},\p_3\] f   }_0+\norm{Z^{m}\eta}_0\norm{\nabla_\a ( \pa^\a_i f)}_{L^\infty}+\norm{\[Z^{m},\p_3\] \eta_k   }_0\norm{\a_{i3}   \a_{kj}\pa_j f}_{L^\infty}
\nonumber
\\&\quad\leq  P(\norm{\eta}_{m})\norm{f}_m.
\end{align}
By the standard commutator estimates, one obtains
\begin{align}\label{Calpha1}
\norm{[Z^{m}, \a_{ij}, \partial_j f]}_0 \ls\norm{Z\a}_{0,m-2}\norm{Z\nabla f}_{L^\infty}+\norm{Z\a}_{L^\infty}\norm{Z\nabla f}_{0,m-2}\leq  P(\norm{\eta}_{m})\norm{f}_m
\end{align}
and
\begin{align}\label{Calpha3}
\norm{\left[Z^{m-1}, \a_{i\ell}\a_{kj}\right]Z\pa_\ell \eta_k\pa_j f}_0 \le \norm{\left[Z^{m-1}, \a_{i\ell}\a_{kj}\right]Z\pa_\ell \eta_k}_0\norm{\nabla f}_{L^{\infty} }\leq  P(\norm{\eta}_{m})\norm{f}_3.
\end{align}
Then \eqref{comest} follows.
\end{proof}

Define the   good unknowns:
\begin{equation}
\label{gun1}
 \mathcal{V}^m = Z^m  v-  Z^m \eta \cdot\nabla_\a  v,\  \mathcal{Q}^m = Z^m  q-  Z^m \eta \cdot\nabla_\a  q
\end{equation}
and
\begin{equation}
\label{gun2}
 \mathcal{B}^m = Z^m b-  Z^m \eta_0^\delta \cdot\nabla_{\a_0^\delta } b.
\end{equation}

\begin{lem}
It holds that
\begin{equation}\label{MHDveermm}
\begin{cases}
 \frac{1}{\gamma p} \dt  \Q- \frac{1}{\gamma p}b \cdot \dt \B + \diva \V =
F^{1,m}
 &\text{in  }\Omega\\
\rho\partial_t\V  +\naba  \Q-b\cdot\naba \B-\eps \da \mathcal{V}^m=F^{2,m} +Z^m\Psi^\epsd      &\text{in } \Omega\\
\partial_t\B- \frac{b}{\gamma p} \dt \Q+   \frac{b}{\gamma p}  b\cdot \dt \B-b\cdot\naba \V=F^{3,m}  &\text{in } \Omega ,
 \end{cases}
\end{equation}
where $F^{1,m}$, $F^{2,m}$ and $F^{3,m}$ are defined by \eqref{ceqValpha02}--\eqref{eqValpha03}, respectively, and that
\begin{equation}\label{MHDveermm2}
\begin{split}
 &\jump{\mathcal{Q}^m }=- J^{-1} Z^m \eta\cdot\n\jump{ \p_3 q},\quad \jump{   \mathcal{B}^m }=-  (J_0^\delta )^{-1} Z^m \eta_0^\delta \cdot\n_0^\delta\jump{ \p_3b} ,
 \\ &\jump{\mathcal{V}^m} =0,\quad \jump{\p_3   \mathcal{V}^m }=-  Z^m \eta \cdot \jump{\p_3(\nabla_\a  v)}  \text{on }\Sigma \text{ and }
\mathcal{V}^m  =0     \text{ on }\Sigma_\pm.
 \end{split}
\end{equation}
Moreover,
\begin{align} \label{qesm0}
\ns{F^{1,m}}_0   +\ns{F^{2,m}}_0+\ns{F^{3,m}}_0  \le \me(  \mathfrak{E}_{m}^{\varepsilon } + \mathfrak{D}_{m}^{\varepsilon } ).
\end{align}
\end{lem}
\begin{proof}
First,
\eqref{commf} and \eqref{gun1} imply
\begin{equation}
Z^m \diva v=  \diva  \mathcal{V}^m+  \mathcal{C}_i^m(v_i) ,
\end{equation}
\begin{align}
Z^m \nabla_\a q=\nabla_\a  \mathcal{Q}^m  + \mathcal{C}^m(q) ,
\end{align}
\begin{align}
  Z^m\da v&=Z^m\pa^\a_i\pa^\a_i v
  = \pa^\a_i\left(Z^m \pa^\a_i v- Z^{m}\eta\cdot\naba \pa^\a_i v\right) + \mathcal{C}_i^m(\pa^\a_i v)\nonumber
\\&=\da \mathcal{V}^m+    \pa^\a_i\left(\mathcal{C}_i^m(v) - Z^{m}\eta\cdot\naba \pa^\a_i v\right) +  \mathcal{C}_i^m(\pa^\a_i v)
\end{align}
and
\beq
Z^m(b\cdot\nabla_{\a} v )  =\[Z^m,b\]\cdot\nabla_{\a} v  + b\cdot Z^m(\nabla_{\a} v )
= \[Z^m,b\]\cdot\nabla_{\a} v + b\cdot\nabla_{\a}  \mathcal{V}^m +b\cdot \mathcal{C}^m(v) .
 \eeq
Using Cauchy's integral  \eqref{re00} and \eqref{jequ1}, one has
\beq\label{gun3}
b\cdot\naba  \equiv\a^Tb\cdot\nabla=J^{-1} J_0^\delta (\a_0^\delta )^Tb_0^\delta  \cdot \nabla  =   \rho (\rho_0^\delta )^{-1}  b_0^\delta  \cdot \nabla_{\a_0^\delta }  .
\eeq
Then denoting $\mathcal{C}_0^{m,\delta}$ for $\mathcal{C}^m$ in \eqref{commfno} with $\eta$ replaced by $\eta_0^\delta $, by \eqref{gun3}, \eqref{commf} and \eqref{gun2}, one obtains
\begin{align}
Z^m (b\cdot\naba b)&=\[Z^m ,\rho (\rho_0^\delta )^{-1}  b_0^\delta \] \cdot \nabla_{\a_0^\delta }b+ \rho (\rho_0^\delta )^{-1}  b_0^\delta  \cdot Z^m(\nabla_{\a_0^\delta }b)\nonumber
\\&=  \[Z^m,\rho (\rho_0^\delta )^{-1} b_0^\delta \] \cdot \nabla_{\a_0^\delta } b
+  b\cdot\naba \mathcal{B}^m+\rho (\rho_0^\delta )^{-1} b_0^\delta \cdot \mathcal{C}_0^{m,\delta}(b) .
\end{align}
Hence, applying $ Z^m $ to  \eqref{MHDv0s} implies  \eqref{MHDveermm} with
\begin{align}\label{ceqValpha02}
F^{1,m}  &=-[Z^m,  \frac{1}{\gamma p} ]\dt   q+[Z^m, \frac{1}{\gamma p}      b] \cdot \dt  b-  \frac{1}{\gamma p}  \dt ( Z^m \eta \cdot\nabla_\a  q )\nonumber\\&\quad+ \frac{1}{\gamma p}      b \cdot  ( Z^m \eta_0^\delta \cdot\nabla_{\a_0^\delta }  \dt b )- \mathcal{C}_i^m(v_i) ,
\end{align}
\begin{align}\label{ceqValpha01}
F^{2,m} & =-[Z^m, \rho]\dt   v-\rho \dt\left( Z^m \eta \cdot\nabla_\a  v\right) - \mathcal{C}^m(q) +\left[ Z^m ,   \rho (\rho_0^\delta )^{-1} b_0^\delta \right]\cdot \nabla_{\a_0^\delta } b
\nonumber\\&\quad  +   \rho (\rho_0^\delta )^{-1}b_0^\delta \cdot \mathcal{C}_{0}^{m,\delta}(b) +\eps   \pa^\a_i\left(\mathcal{C}_i^m(v) - Z^{m}\eta\cdot\naba \pa^\a_i v\right) +\eps  \mathcal{C}_i^m(\pa^\a_i v)
\end{align}
and
\begin{align} \label{eqValpha03}
F^{3,m}= & [Z^m,  \frac{b}{\gamma p}]\dt   q-[Z^m, \frac{b}{\gamma p}  b] \cdot \dt  b    -Z^m \eta_0^\delta \cdot\nabla_{\a_0^\delta }\dt b+ \frac{b}{\gamma p} \dt\left( Z^m \eta \cdot\nabla_\a  q\right)\nonumber\\&- \frac{b}{\gamma p}  b\cdot (Z^m \eta_0^\delta \cdot\nabla_{\a_0^\delta }\dt b)+ \[Z^m,b\]\cdot\nabla_{\a} v+b\cdot \mathcal{C}^m(v).
\end{align}

Next,  \eqref{MHDv0sb} and \eqref{gun1}--\eqref{gun2}  imply
\beq
\jump{\mathcal{Q}^m} =  -  Z^m \eta \cdot \jump{\nabla_\a  q}=- Z^m \eta_j\a_{j3}\jump{ \p_3 q}=- J^{-1} Z^m \eta\cdot\n\jump{ \p_3 q}\text{ on }\Sigma,
\eeq
 \beq
 \jump{   \mathcal{B}^m }=-  Z^m \eta_0^\delta \cdot \jump{\nabla_{\a_0^\delta } b}=-(J_0^\delta )^{-1} Z^m \eta_0^\delta \cdot\n_0^\delta  \jump{ \p_3b} \text{ on }\Sigma,
 \eeq
\beq
  \jump{    \mathcal{V}^m }=0,\quad\jump{\p_3   \mathcal{V}^m }=-  \jump{\p_3(Z^m \eta \cdot    \nabla_\a  v)}=-  Z^m \eta \cdot \jump{\p_3(\nabla_\a  v)}\text{ on }\Sigma
 \eeq
 and
 \begin{equation}
 \mathcal{V}^m =-Z^m\eta_j\a_{j3}\p_3 v=-J^{-1}Z^m\eta_3\p_3 v= 0\text{ on }\Sigma_\pm.\end{equation}
Thus \eqref{MHDveermm2} follows.

Now one estimates $F^{1,m}$, $F^{2,m}$ and $F^{3,m}$. It follows similarly as for \eqref{comest} and the definitions \eqref{edef} and \eqref{ddef} of  $\mathfrak{E}_{m}^{\varepsilon } $ and $ \mathfrak{D}_{m}^{\varepsilon } $ that, since $\dt\eta=v $ and $m\ge 4$,
\beq
\ns{F^{1,m}}_0 +\ns{F^{3,m}}_0 \le \me(  \mathfrak{E}_{m}^{\varepsilon } + \mathfrak{D}_{m}^{\varepsilon } )
 \eeq
 and
\beq
\ns{F^{2,m}}_0 \le
 \me \(\mathfrak{E}_{m}^{\varepsilon } + \mathfrak{D}_{m}^{\varepsilon }+\eps^2\norm{\eta}_{m+1}^2+\eps^2\norm{v}_{m+1}^2\)   \le \me(  \mathfrak{E}_{m}^{\varepsilon } + \mathfrak{D}_{m}^{\varepsilon } ).
\eeq
One then concludes \eqref{qesm0}.
\end{proof}

Now we state the estimates of the highest order tangential spatial derivatives.
\begin{prop}\label{tane}
For $t\in [0,T_0^\epsd]$ with $T_0^\epsd\le 1$, it holds that
\begin{align} \label{spaes}
 \ns{(p, v, b)(t) }_{0,m} +            \as{   \eta(t)}_m+\eps\int_0^t \norm{  v }_{1,m}^2 \le \me(t)\(  \Mm+\eps\Mmd  +  t^{1/2}\fg{m}^\eps(t)       \).
\end{align}
\end{prop}
\begin{proof}
Taking the $L^2(\Omega)$ inner product of the   equations in \eqref{MHDveermm} with $J\Q$,  $J\mathcal{V}^m $  and  $J\mathcal{B}^m$, respectively,  adding the resulting together, and then integrating by parts in $t$ and by a simple combination, one deduces that, since $\dt(\rho J)=0$,
\begin{align}\label{ttt1}
& \dfrac{1}{2}\dfrac{d}{dt}\int_{\Omega} J(\frac{1}{\gamma p} \abs{\Q -b\cdot \B}^2 +\rho \abs{\mathcal{V}^m }^2 + \abs{\mathcal{B}^m }^2 )
 \nonumber
\\&\quad    +\int_{\Omega}J \diva ( \mathcal{Q}^m   \mathcal{V}^m)- \int_\Omega Jb\cdot\naba (\mathcal{B}^m \cdot \mathcal{V}^m)-\eps \int_{\Omega}J  \da \mathcal{V}^m\cdot \mathcal{V}^m
 \nonumber
\\&\quad= \int_{\Omega} \Big(\hal\dt( \frac{J}{\gamma p}) (\abs{\Q}^2+\abs{   b\cdot  \B}^2)+ \frac{J}{\gamma p}  \dt b\cdot  \B b\cdot \B  +\hal\dt J  \abs{     \B}^2 -\dt( \frac{Jb}{\gamma p} ) \cdot \B \Q \Big)
 \nonumber
\\&\qquad  +\int_\Omega J (  F^{1,m}  \Q +   F^{2,m}\cdot \mathcal{V}^m +  F^{3,m}\cdot \mathcal{B}^m)+\int_\Omega JZ^m\Psi^\epsd   \cdot \mathcal{V}^m.
\end{align}
By the definitions of $\mathcal{Q}^m,\ \mathcal{B}^m$ and $\mathcal{V}^m$,
\begin{align}\label{baba1}
 &\int_{\Omega} \Big(\hal\dt( \frac{J}{\gamma p}) (\abs{\Q}^2+\abs{   b\cdot  \B}^2)+ \frac{J}{\gamma p}  \dt b\cdot  \B b\cdot \B  +\hal\dt J  \abs{     \B}^2 -\dt( \frac{Jb}{\gamma p} ) \cdot \B \Q \Big)\nonumber
\\&\quad \le \me \(\ns{\Q}_0+\ns{\B}_0+\norm{\B}_0\norm{\Q}_0\) \le \me  \mathfrak{E}_{m}^{\varepsilon },
\end{align}
and by  \eqref{qesm0} and \eqref{psies},
\begin{align}\label{ttt111}
&\int_\Omega J (  F^{1,m}  \Q +   F^{2,m}\cdot \mathcal{V}^m+Z^m\Psi^\epsd   \cdot \mathcal{V}^m +  F^{3,m}\cdot \mathcal{B}^m) \nonumber
\\&\quad\le \me \(\norm{F^{1,m}}_0\norm{\Q}_0+\norm{F^{2,m}}_0\norm{\V}_0+\norm{Z^m\Psi^\epsd}_0\norm{\mathcal{V}^m}_0+\norm{F^{3,m}}_0\norm{\B}_0 \)
\nonumber
\\&\quad \le \me \( \sqrt{ \mathfrak{E}_{m}^{\varepsilon }+ \mathfrak{D}_{m}^{\varepsilon } } \sqrt{\mathfrak{E}_{m}^{\varepsilon }}+ \eps \Mmd\sqrt{\mathfrak{E}_{m}^{\varepsilon }}\).
\end{align}

Now we turn to estimate the left hand side of \eqref{ttt1}.
First,
integrating by parts over $\Omega_\pm$ and using   \eqref{MHDveermm2}, one obtains
\begin{align}\label{ttt100056}
 -\eps \int_{\Omega}J  \da \mathcal{V}^m\cdot \mathcal{V}^m
& =\eps \int_{\Sigma}   \jump{\naba\mathcal{V}^m}\n\cdot  \mathcal{V}^m+\eps \int_{\Omega}J  \abs{\nabla_\a \mathcal{V}^m}^2\nonumber
\\& =-  \eps \int_{\Sigma} J^{-1} |\n|^2 Z^m \eta\cdot \jump{\p_3(\naba v_i)}\mathcal{V}_i^m+\eps \int_{\Omega}J  \abs{\nabla_\a \mathcal{V}^m}^2  .
\end{align}
By the trace theory, one gets
\begin{align}\label{ttt10020}
 \eps \int_{\Sigma} J^{-1} |\n|^2 Z^m \eta\cdot \jump{\p_3(\naba v_i)}\mathcal{V}_i^m&\le  \eps \me \abs{Z^m \eta}_{0}\abs{\mathcal{V}^m}_{0} \le \eps  \me \sqrt{\fe{m}^{\varepsilon }}\norm{\mathcal{V}^m}_{1}.
\end{align}
Next, one estimates the most delicate remaining two terms. Integrating by parts over $\Omega_\pm$ and using \eqref{MHDveermm2}, one has
\begin{align}
\label{gg2}
 -\int_{\Omega}J \diva ( \mathcal{Q}^m   \mathcal{V}^m)  =\int_{\Sigma }\jump{ \mathcal{Q}^m}  \mathcal{V}^m\cdot\n  =  -\int_{\Sigma } \jump{\pa_3 q} J^{-1} Z^m \eta\cdot\n \mathcal{V}^m\cdot\n .
\end{align}
By the definition  of $\mathcal{V}^m$ and since $\dt\eta=v$,  integrating by parts in $t$ yields
\begin{align} \label{reg1}
& -\int_{\Sigma } \jump{\pa_3 q} J^{-1} Z^m \eta\cdot\n \mathcal{V}^m\cdot\n    =-\int_{\Sigma }\jump{\pa_3 q}J^{-1} Z^m \eta\cdot\n  ( Z^m  v -  Z^m   \eta \cdot\nabla_\a  v )\cdot\n\nonumber
  \\&\quad=-\int_{\Sigma }\jump{\pa_3 q}J^{-1} Z^m \eta\cdot\n    Z^m  \dt\eta  \cdot\n+ \int_{\Sigma }\jump{\pa_3 q}J^{-1} Z^m \eta\cdot\n   (Z^m   \eta \cdot\nabla_\a  v)  \cdot\n
\nonumber
  \\&\quad=-\hal \frac{d}{dt}\int_{\Sigma }\jump{\pa_3 q}J^{-1} \abs{Z^m \eta\cdot\n}^2  +\hal \int_{\Sigma }\dt(\jump{\pa_3 q}J^{-1} )\abs{Z^m \eta\cdot\n}^2 \nonumber\\
&\qquad
+\int_{\Sigma }\jump{\pa_3 q}J^{-1} Z^m \eta\cdot\n    Z^m   \eta  \cdot \dt\n + \int_{\Sigma }\jump{\pa_3 q}J^{-1} Z^m \eta\cdot\n   (Z^m   \eta \cdot\nabla_\a  v ) \cdot\n
\nonumber
  \\&\quad\le-\hal \frac{d}{dt}\int_{\Sigma }\jump{\pa_3 q}J^{-1} \abs{Z^m \eta\cdot\n}^2+\me  \as{ Z^m \eta}_0
\nonumber
  \\&\quad\le-\hal \frac{d}{dt}\int_{\Sigma }\jump{\pa_3 q}J^{-1} \abs{Z^m \eta\cdot\n}^2+
 \me \fe{m}^{\eps }.
\end{align}
Similarly,  by using the Piola identity, Proposition \ref{prop1} and \eqref{MHDveermm2}, one deduces
\begin{align}\label{troublesome}
&\int_\Omega Jb\cdot\naba (\mathcal{B}^m \cdot \mathcal{V}^m)= - \int_\Sigma b \cdot \n   \jump{\mathcal{B}^m} \cdot \mathcal{V}^m-\int_\Omega J \diva b \mathcal{B}^m \cdot \mathcal{V}^m
 \nonumber
\\ &\quad = - \int_\Sigma b_0^\delta \cdot \n_0^\delta   \jump{\mathcal{B}^m} \cdot ( Z^m v  -Z^m \eta \cdot\nabla_\a  v  )-\int_\Omega J_0^\delta \diverge_{\a_0^\delta} b_0^\delta \mathcal{B}^m \cdot \mathcal{V}^m
 \nonumber
\\ &\quad= \int_\Sigma   b_0^\delta \cdot \n_0^\delta (J_0^\delta )^{-1} Z^m \eta_0^\delta \cdot\n_0^\delta  \jump{ \p_3b} \cdot ( Z^m \partial_t\eta  -Z^m \eta \cdot\nabla_\a  v  ) -\int_\Omega J_0^\delta \diverge_{\a_0^\delta} b_0^\delta \mathcal{B}^m \cdot \mathcal{V}^m\nonumber
\\ &\quad= \dtt\int_\Sigma b_0^\delta \cdot \n_0^\delta (J_0^\delta )^{-1} Z^m \eta_0^\delta \cdot\n_0^\delta  \jump{ \p_3b} \cdot   Z^m \eta-\int_\Sigma b_0^\delta \cdot \n_0^\delta (J_0^\delta )^{-1} Z^m \eta_0^\delta \cdot\n_0^\delta  \jump{ \p_3 \dt b} \cdot   Z^m \eta \nonumber
\\ &\qquad   -\int_\Sigma   b_0^\delta \cdot \n_0^\delta (J_0^\delta )^{-1} Z^m \eta_0^\delta \cdot\n_0^\delta  \jump{ \p_3b} \cdot  Z^m \eta \cdot\nabla_\a  v -\int_\Omega J_0^\delta \diverge_{\a_0^\delta} b_0^\delta \mathcal{B}^m \cdot \mathcal{V}^m  \nonumber
\\&\quad\le    \dtt\int_\Sigma b_0^\delta \cdot \n_0^\delta (J_0^\delta )^{-1} Z^m \eta_0^\delta \cdot\n_0^\delta  \jump{ \p_3b} \cdot   Z^m \eta +\me  \(\abs{ Z^m \eta_0^\delta }_0 \abs{ Z^m \eta}_0+\norm{\mathcal{B}^m}_0 \norm{ \mathcal{V}^m}_0 \)
 \nonumber
\\&\quad\le   \dtt\int_\Sigma b_0^\delta \cdot \n_0^\delta (J_0^\delta )^{-1} Z^m \eta_0^\delta \cdot\n_0^\delta  \jump{ \p_3b} \cdot   Z^m \eta+\me\(\sqrt{ \fe{m}^{\eps }(0)}\sqrt{ \fe{m}^{\eps }}+\fe{m}^{\eps }\).
 \end{align}

As a consequence of the estimates \eqref{baba1}--\eqref{troublesome}, one deduces from \eqref{ttt1} that
\begin{align} \label{ggdddsa}
& \dfrac{1}{2}\dfrac{d}{dt}\(\int_{\Omega} J(\frac{1}{\gamma p} \abs{\Q -b\cdot \B}^2 +\rho \abs{\mathcal{V}^m }^2 + \abs{\mathcal{B}^m }^2 )-\i_m\)  +\eps \int_{\Omega}J \abs{ \naba \mathcal{V}^m}^2\nonumber
 \\&\quad\le  \me\(\sqrt{\fe{m}^{\varepsilon }(0)+ \fe{m}^{\varepsilon }+ \mathfrak{D}_{m}^{\varepsilon } } \sqrt{\mathfrak{E}_{m}^{\varepsilon }  } + \eps \Mmd\sqrt{\mathfrak{E}_{m}^{\varepsilon }}+ \eps    \sqrt{\fe{m}^{\varepsilon }}\norm{\mathcal{V}^m}_{1}\),
\end{align}
where
\begin{align}\label{ggdddsa2}
\i_m&:=-\dfrac{1}{2} \int_{\Sigma }\jump{\pa_3 q}J^{-1} \abs{Z^m \eta\cdot\n}^2  +\int_\Sigma b_0^\delta \cdot \n_0^\delta (J_0^\delta )^{-1} Z^m \eta_0^\delta \cdot\n_0^\delta  \jump{ \p_3b} \cdot   Z^m \eta
\nonumber
  \\& \le \me  \(\as{ Z^m \eta}_0+ \abs{ Z^m \eta_0^\delta }_0 \abs{ Z^m \eta}_0  \) .
\end{align}
Integrating \eqref{ggdddsa} directly in time,  by \eqref{ggdddsa2}, \eqref{gesa1} and using the following inequality
 \beq\label{nablaes}
\norm{ \nabla f}_0^2\ls \int_{\Omega}J \abs{ \naba f}^2,
 \eeq
 Poincar\'e's inequality, Cauchy's inequality and the Cauchy-Schwarz inequality, one obtains
\begin{align} \label{es11d}
 & \ns{(\Q ,\mathcal{V}^m,\mathcal{B}^m)(t) }_0+\eps\int_0^t \norm{ \mathcal{V}^m}_1^2\nonumber
\\&\quad\le \me(t)\( \Mm+\eps^2\Mmd  +  \as{Z^m\eta(t) }_0 + \int_0^t  \sqrt{  \fe{m}^{\varepsilon }+ \mathfrak{D}_{m}^{\varepsilon }   } \sqrt{\mathfrak{E}_{m}^{\varepsilon } }   \)
\nonumber
\\&\quad\le \me(t)\( \Mm+\eps^2\Mmd  +   \as{Z^m\eta (t)}_0 +  t\fg{m}^\eps(t) +\sqrt{\fg{m}^\eps(t)}\int_0^t  \sqrt{  \mathfrak{D}_{m}^{\varepsilon }  }      \)
\nonumber
\\&\quad\le\me(t)\(  \Mm+\eps^2\Mmd  +   \as{Z^m\eta(t) }_0 +     t^{1/2}\fg{m}^\eps(t)  \)  .
\end{align}
By the definitions of good unknowns again,   \eqref{es11d} and \eqref{etaest}, one deduces
\begin{align} \label{es11d21}
 &\ns{(Z^m p,Z^m v,Z^m b )(t)}_0 +\eps\int_0^t \norm{ Z^m v}_1^2
\nonumber\\ &\quad\le \me(t)\(\mathfrak{F}_{m}^{\varepsilon}(t)+\ns{(\Q ,\mathcal{V}^m,\mathcal{B}^m)(t) }_0+\ns{Z^m\eta(t)}_0+\eps\int_0^t \Big(\norm{ \V}_1^2+\norm{ Z^m\eta}_1^2\Big)\)
\nonumber\\ &\quad \le\me(t)\(  \Mm+\eps\Mmd  +  \as{Z^m\eta(t) }_0 +     t^{1/2}\fg{m}^\eps(t) \)   .
\end{align}

The key point here now is to use the term $ \ns{(Z^m p,Z^m b) }_0$ in the left hand side of \eqref{es11d21} to control the  term $\as{Z^m\eta }_0$ in the right hand side. Recalling Cauchy's integral  \eqref{re00}, then by \eqref{es11d21} and \eqref{etaest}, one has
\begin{align}\label{kes1}
 &\norm{Z^m (b_0^\delta \cdot\nabla_{\a_0^\delta }\eta) }_0^2=\norm{Z^m (\rho_0^\delta  \rho^{-1}b) }_0^2
\le \me\(\mathfrak{F}_{m}^{\varepsilon}+ \ns{(Z^m p,Z^m b)}_0 \)
 \nonumber\\&\quad\le \me\(  \Mm+\eps\Mmd  +  \as{Z^m\eta }_0 +     t^{1/2}\fg{m}^\eps  \).
\end{align}
Introducing further the good unknown
\beq\label{defer}
\Xi^m:=Z^m \eta-  Z^m \eta_0^\delta \cdot\nabla_{\a_0^\delta } \eta,
\eeq
then by \eqref{commf} and \eqref{defer}, one obtains
\begin{align}\label{kes2}
Z^m (b_0^\delta \cdot\nabla_{\a_0^\delta }\eta)&=\[Z^m,b_0^\delta \]\cdot\nabla_{\a_0^\delta }\eta+b_0^\delta \cdot Z^m\nabla_{\a_0^\delta }\eta
\nonumber
\\&=\[Z^m,b_0^\delta \]\cdot\nabla_{\a_0^\delta }\eta+b_0^\delta \cdot\nabla_{\a_0^\delta }  \Xi^m+b_0^\delta \cdot \mathcal{C}_0^{m,\delta}(\eta).
\end{align}
Hence, \eqref{kes2} and \eqref{kes1} imply that, similarly as for \eqref{comest},
\begin{align}\label{kes3}
\ns{b_0^\delta \cdot\nabla_{\a_0^\delta }  \Xi^m }_0&\le \ns{Z^m (b_0^\delta \cdot\nabla_{\a_0^\delta }\eta) }_0+\ns{\[Z^m,b_0^\delta \]\cdot\nabla_{\a_0^\delta }\eta}_0+\ns{b_0^\delta \cdot \mathcal{C}_0^{m,\delta}(\eta)}_0\nonumber
\\& \le\me\( \Mm+\eps\Mmd  +  \as{Z^m\eta }_0 +     t^{1/2}\fg{m}^\eps(t) \) .
\end{align}
Now, since $((\a_0^\delta )^Tb_0^\delta )_3=(J_0^\delta )^{-1} b_0^\delta \cdot \n_0^\delta  \neq 0$ near $\Sigma$, it is crucial to apply Proposition \ref{tra_es} with $\bar B=(\a_0^\delta )^Tb_0^\delta $ and $f=\Xi^m $ to have
\beq\label{kes4}
  \abs{\Xi^m }_0^2 \ls \norm{(\a_0^\delta )^Tb_0^\delta  \cdot \nabla \Xi^m }_0\norm{\Xi^m }_0+  \ns{\Xi^m }_0 .
\eeq
Then by  \eqref{kes1}, \eqref{kes4}, \eqref{kes3}, \eqref{gesa1} and using Cauchy's inequality,  one has that for any $\epsilon>0,$
\begin{align}
  \abs{Z^m \eta}_0^2 &\le  \abs{\Xi^m}_0^2 +\abs{Z^m \eta_0^\delta \cdot\nabla_{\a_0^\delta } \eta}_0^2
  \nonumber
\\&\ls  \norm{(\a_0^\delta )^Tb_0^\delta  \cdot \nabla \Xi^m }_0\norm{\Xi^m }_0+  \ns{\Xi^m }_0+\me \as{Z^m\eta_0^\delta }_0  \nonumber
\\& \ls\sqrt{ \me\(\Mm+\eps\Mmd  +  \as{Z^m\eta }_0 +     t^{1/2}\fg{m}^\eps(t) \) }\norm{\Xi^m }_0+  \ns{\Xi^m }_0+\me \Mm
\nonumber
\\& \le \epsilon \as{Z^m\eta }_0+C_\epsilon \me\(  \Mm+\eps\Mmd  +  \ns{\Xi^m }_0+     t^{1/2}\fg{m}^\eps  \)
\nonumber
\\& \le \epsilon \as{Z^m\eta }_0+C_\epsilon \me\( \Mm+\eps\Mmd  +     t^{1/2}\fg{m}^\eps  \),
\end{align}
which  implies, by taking $\epsilon>0$ sufficiently small,
\begin{align}\label{sbt23}
  \abs{Z^m \eta}_0^2       \le\me\( \Mm+\eps\Mmd  +       t^{1/2}\fg{m}^\eps  \) .
\end{align}
Therefore, by   \eqref{sbt23}, one obtains from \eqref{es11d21} that
\begin{align} \label{es11d23}
 &\ns{(Z^m p,Z^m v,Z^m b)(t) }_0+\eps\int_0^t \norm{ Z^m v}_1^2   \le\me(t)\( \Mm+\eps\Mmd  +     t^{1/2}\fg{m}^\eps(t) \)   .
\end{align}
This gives \eqref{spaes} by recalling \eqref{etaest}.
\end{proof}

Now we consider the estimates of the rest of highest order tangential   derivatives. Let $\al\in \mathbb{N}^{1+3}$ be with $|\al|= m$ and  $\al_0\ge 1$, and one applies  $\bar Z^\al$ to \eqref{MHDv0s} to find
\begin{equation}\label{MHDveer}
\begin{cases}
\frac{1}{\gamma p} \dt \bar Z^\al q-\frac{1}{\gamma p}     b \cdot \dt \bar Z^\al  b + \diva \bar Z^\al v = F^{1,\al}  &\text{in  }\Omega\\
\rho\partial_t \bar Z^\al v  +\naba \bar Z^\al  q-b\cdot\naba \bar Z^\al b-\eps \da \bar Z^\al v= F^{2,\al} + \bar Z^\al\Psi^\epsd   &\text{in } \Omega\\
\partial_t\bar Z^\al b-\frac{b}{\gamma p}   \dt \bar Z^\al q+\frac{b}{\gamma p}    b\cdot \dt \bar Z^\al b-b\cdot\naba \bar Z^\al v=F^{3,\al} &\text{in } \Omega ,
 \end{cases}
\end{equation}
where
\beq
  F^{1,\al} = -\[\bar Z^\al,\frac{1}{\gamma p} \]\dt  q+ \[\bar Z^\al,\frac{1}{\gamma p}     b\] \cdot \dt    b -\[ \bar Z^\al ,\diva\] v,
\eeq
\beq
  F^{2,\al} =  -\[ \bar Z^\al ,\rho\]\partial_t  v -\[ \bar Z^\al ,\naba\] q+\[\bar Z^\al,b\cdot\naba\]b+\varepsilon\[\bar Z^\al,\da\] v
\eeq
and
\beq
  F^{3,\al} =  \[ \bar Z^\al ,\frac{b}{\gamma p} \]\partial_t  q -\[ \bar Z^\al ,\frac{b}{\gamma p} b\] \cdot\dt b+\[\bar Z^\al,b\cdot\naba\]v.
\eeq

 \begin{lem}
 It holds that
 \begin{align} \label{qesa0}
\ns{F^{1,\al}}_0 + \ns{F^{2,\al}}_0 + \ns{F^{3,\al}}_0      \le \me(  \mathfrak{E}_{m}^{\varepsilon } + \mathfrak{D}_{m}^{\varepsilon } ).
\end{align}
\end{lem}
\begin{proof}
\eqref{qesa0} follows similarly as for \eqref{qesm0}.
\end{proof}

Now we estimate such derivatives.

\begin{prop}\label{tanest1}
For $t\in [0,T_0^\epsd]$ with $T_0^\epsd\le 1$, it holds that
\begin{align} \label{tanest1es}
&\sum_{j=1 }^{m} \norm{(\dt^j  p,\dt^j  v,\dt^j  b)(t)}_{0,m-j }^2
+ \eps\int_0^t\sum_{j=1 }^{m}\norm{\dt^j   v}_{1,m-j }^2
\nonumber\\&\quad \le \me(t)\(  \Mm+\eps \Mmd  +     t^{1/2}\fg{m}^\eps(t) \) .
\end{align}
\end{prop}
\begin{proof}
Let $\al\in \mathbb{N}^{1+3}$ be with $|\al|= m$ and  $\al_0\ge 1$. Taking the $L^2(\Omega)$ inner product of the equations in \eqref{MHDveer} for such $\al$ with $J \bar Z^\al q,  J \bar Z^\al v$ and $J \bar Z^\al b$, respectively, similarly as for \eqref{ttt1},
one obtains
\begin{align}\label{ttt123}
& \dfrac{1}{2}\dfrac{d}{dt}\int_{\Omega} J(\frac{1}{\gamma p} \abs{\bar Z^\al q -b\cdot \bar Z^\al b}^2 +\rho \abs{\bar Z^\al v }^2 + \abs{\bar Z^\al b}^2 )
 \nonumber
 \\&\quad  +\int_{\Omega}J \diva ( \bar Z^\al q \cdot \bar Z^\al v)- \int_\Omega Jb\cdot\naba (\bar Z^\al b \cdot \bar Z^\al v)-\eps \int_{\Omega}J  \da \bar Z^\al v\cdot \bar Z^\al v
 \nonumber
\\&\quad = \int_{\Omega} \!\(\!\hal\dt( \frac{J}{\gamma p}) \!\Big( \!\abs{\bar Z^\al q}^2+\abs{b\cdot \bar Z^\al b}^2\!\Big)\!+  \frac{J}{\gamma p}  \dt b\cdot  \bar Z^\al b b\cdot \bar Z^\al b +\hal\dt  J  \abs{     \bar Z^\al b}^2-\dt(  \frac{Jb}{\gamma p} ) \cdot \bar Z^\al b \bar Z^\al q \!\)
 \nonumber
\\&\qquad  +\int_\Omega J (  F^{1,\al}\cdot \bar Z^\al q +   F^{2,\al}\cdot \bar Z^\al v +  F^{3,\al}\cdot \bar Z^\al b)+\int_\Omega \bar Z^\al \Psi^\epsd   \cdot \bar Z^\al v.
\end{align}
Integrating by parts over $\Omega_\pm$ and using the boundary conditions in \eqref{MHDv0sb}, one gets
\begin{align}\label{ttt1000ga}
&-\int_{\Omega}J \diva ( \bar Z^\al q \cdot \bar Z^\al v)+ \int_\Omega Jb\cdot\naba (\bar Z^\al b \cdot \bar Z^\al v)
\nonumber\\&\quad= -\int_\Omega J\diva b \bar Z^\al b \cdot \bar Z^\al v \le \me \norm{\bar Z^\al b}_0\norm{\bar Z^\al v}_0 \le \me \fe{m}^\eps
\end{align}
and
\beq
 -\eps \int_{\Omega}J  \da \bar Z^\al  v\cdot \bar Z^\al  v
  = \eps \int_{\Omega}J  \abs{\nabla_\a \bar Z^\al  v}^2  .
\eeq
One has directly
\begin{align}
 & \int_{\Omega} \(\hal\dt( \frac{J}{\gamma p}) \Big( \abs{\bar Z^\al q}^2+\abs{b\cdot \bar Z^\al b}^2\Big)+ \frac{J}{\gamma p}  \dt b\cdot  \bar Z^\al b b\cdot \bar Z^\al b +\hal\dt  J  \abs{     \bar Z^\al b}^2-\dt( \frac{Jb}{\gamma p} ) \cdot \bar Z^\al b \bar Z^\al q \)\nonumber
\\&\quad \le  \me \(\ns{\bar Z^\al q}_0+\ns{\bar Z^\al b}_0+\norm{\bar Z^\al q}_0\norm{\bar Z^\al b}_0\) \le \me  \mathfrak{E}_{m}^{\varepsilon }  .
\end{align}
By \eqref{qesa0} and  \eqref{psies}, one has directly
\begin{align}\label{gg2a}
&\int_\Omega J (  F^{1,\al}\cdot \bar Z^\al q +   F^{2,\al}\cdot \bar Z^\al v +\bar Z^\al\Psi^\epsd   \cdot \bar Z^\al v+  F^{3,\al}\cdot \bar Z^\al b)\nonumber
\\&\quad\le  \me \(\norm{F^{1,\al}}_0\norm{\bar Z^\al q }_0+\norm{F^{2,\al}}_0\norm{\bar Z^\al v}_0+\norm{\bar Z^\al\Psi^\epsd }_0\norm{\bar Z^\al v}_0+\norm{F^{3,\al}}_0\norm{\bar Z^\al b}_0\)
\nonumber
\\&\quad\le   \me \( \sqrt{ \mathfrak{E}_{m}^{\varepsilon }+ \mathfrak{D}_{m}^{\varepsilon } }\sqrt{\mathfrak{E}_{m}^{\varepsilon }}+\eps \Mmd\sqrt{\mathfrak{E}_{m}^{\varepsilon }}\) .
\end{align}

As a consequence of the estimates \eqref{ttt1000ga}--\eqref{gg2a}, one deduces from \eqref{ttt123} that
\begin{align} \label{gdes111223}
&  \dfrac{1}{2}\dfrac{d}{dt}\int_{\Omega} J(\frac{1}{\gamma p} \abs{\bar Z^\al q -b\cdot \bar Z^\al b}^2 +\rho \abs{\bar Z^\al v }^2 + \abs{\bar Z^\al b}^2 )+\eps  \int_{\Omega}J \abs{ \nabla_\a \bar Z^\al v}^2 \nonumber
 \\&\quad\le\me \(\sqrt{ \mathfrak{E}_{m}^{\varepsilon }+ \mathfrak{D}_{m}^{\varepsilon } }\sqrt{\mathfrak{E}_{m}^{\varepsilon }}+\eps \Mmd\sqrt{\mathfrak{E}_{m}^{\varepsilon }}\).
\end{align}
Integrating \eqref{gdes111223} directly in time, by \eqref{nablaes}, \eqref{psies} and using Poincar\'e's inequality and H\"older's inequality, one obtains
\begin{align} \label{qqqw}
    \ns{(\bar Z^\al p,\bar Z^\al v ,\bar Z^\al b)(t)}_0  +\eps\int_0^t   \norm{   \bar Z^\al v}_1^2 &\le  \Mm+\eps\Mmd  + \me(t)\int_0^t \sqrt{ \mathfrak{E}_{m}^{\varepsilon }+ \mathfrak{D}_{m}^{\varepsilon } }\sqrt{\mathfrak{E}_{m}^{\varepsilon }}
 \nonumber
\\&\le  \Mm+\eps\Mmd  + \me(t)  t^{1/2}\fg{m}^\eps(t) .
\end{align}
This gives \eqref{tanest1es} by summing over such $\al$  and recalling \eqref{etaest}.
\end{proof}

\section{Normal derivatives estimates} \label{lasts4}

In this section, we will estimate the normal derivatives of $(p,v,b)$. In light of the tangential estimates in Section \ref{lasts5}, it suffices to derive the estimates near the boundary $\Sigma\cup\Sigma_\pm$, and thus the analysis in the following will be   carried out mostly near $\Sigma\cup\Sigma_\pm$, without mentioning explicitly. The fact that $(\a ^Tb )_3=J^{-1}b_0^\delta \cdot\n_0^\delta  \neq0$ in such region will be used crucially.

\begin{prop}\label{tane23}
For $t\in [0,T_0^\epsd]$ with $T_0^\epsd\le 1$, it holds that
\begin{align} \label{total1es41}
  \fe{m}^{\varepsilon}(t)+\int_0^t \fdb{m}^{\varepsilon}
 \le  \me(t)  \( \Mm+\eps\Mmd  + t^{1/2}\fg{m}^\eps(t)    \).
\end{align}
\end{prop}
\begin{proof}
 First, \eqref{dividen22} implies
\beq\label{dee}
\p_3 b\cdot \n= -J\a_{i\beta}\p_\beta b_i+J_0^\delta\diverge_{\a_0^\delta} b_0^\delta.
\eeq
Then by \eqref{fadl1} and \eqref{gesa1}, one obtains that for $j=0,\dots,m-1$,
\begin{align}\label{vve1}
\norm{  \dt^{j} \p_3 b\cdot \n}_{0,m-j -1}^2 \le\me \(\mathfrak{F}_{m}^{\varepsilon} +\norm{  \dt^{j} b}_{0,m-j }^2\)+\Mm.
\end{align}

Next, the second equation and fourth equation in \eqref{MHDve} imply
 \beq\label{tbb}
b\cdot \naba    v=\dt b - \frac{b}{\gamma p}\dt  p.
\eeq
Thus,  since $(\a ^Tb )_3 \neq0$ near $\Sigma\cup\Sigma_\pm$,  one obtains
\beq \label{nneq1}
\p_3 v=\frac{1}{(\a^Tb)_3}(b\cdot \naba    v-(\a^Tb)_\beta\p_\beta v)=\frac{1}{(\a^Tb)_3}\(\partial_tb-  \frac{b}{\gamma p}\dt  p-(\a^Tb)_\beta\p_\beta v\).
\eeq
One then deduces that for $j=0,\dots,m-1$,
\begin{align}\label{vve12}
\norm{ \p_3 \dt^{j} v}_{0,m-j -1}^2 \le\me \(\mathfrak{F}_{m}^{\varepsilon} +\norm{  ( \dt^{j+1} p,   \dt^{j+1} b )}_{0,m-j -1}^2 +\norm{  \dt^{j} v}_{0,m-j }^2\).
\end{align}
This, together with \eqref{vve1}, Propositions \ref{tres1}, \ref{tane} and \ref{tanest1}, yields \eqref{total1es41}.
\end{proof}

Next,  in terms of the orthogonal base:
\beq\label{base}
\left\{\tau^1 =\frac{\p_1\eta}{|\p_1\eta|},\tau^2 =\frac{\tilde \tau}{|\tilde \tau|}, n=\frac{\n}{|\n|} \right\}\text{ with }\tilde \tau=\p_2\eta-\frac{\p_1\eta\cdot\p_2\eta}{|\p_1\eta|^2}\p_1\eta,
\eeq
one has the following expressions for $\p_3b\cdot\tau^\beta,\ \beta=1,2,$ and $\p_3 p$.
\begin{lem}
Let $\al\in \mathbb{N}^{1+3}$ with $\abs{\al}\le m-1$.
It holds that  near $\Sigma\cup\Sigma_\pm$,
\begin{align}\label{teqq}
  \p_3 \p^\al b\cdot\tau^\beta  +\eps \frac{ \abs{\n}^2}{(J\a^Tb)_3^2}   \p_3 \p^\al\( b\cdot\naba v\)\cdot\tau^\beta   =  G _{\tau^\beta }^\al,\ \beta=1,2
\end{align}
and
\beq\label{teqq2}
   \p_3 \p^\al p+ \p_3 \p^\al b\cdot\tau^\beta   b\cdot \tau^\beta   -\varepsilon \p_3 \p^\al \diva v = G _n^\al.
\eeq
Here  $ G _{\tau^\beta }^\al $ and $G _n^\al$ are defined by  \eqref{eeqq2} and \eqref{eeqq3},  respectively, which satisfy
\begin{align}\label{teqqes2}
\ns{G _{\tau^\beta }^\al}_0+\ns{G _n^\al}_0   &\le \me\Big( \ns{\p^\al\Psi^\epsd}_0+\mathfrak{F}_{m}^{\varepsilon}+\ns{(\p^\al p,\p^\al b)}_{0,1}+\ns{\p^\al \dt v}_{0}
\nonumber\\&\qquad\quad+\eps^2(\ns{\p^\al v}_{1,1}+\fe{m}^\eps+\fd{m}^\eps)\Big).
\end{align}
\end{lem}
\begin{proof}

One first projects  the third equation in \eqref{MHDve} along $\tau^\beta $, $\beta=1,2$. For this, one writes
\beq
 \Delta_\a v= \a_{k3} \a_{k3} \p_3^2 v+\a_{k3}\p_3 \a_{k3} \p_3 v+\sum_{i+j<6}\a_{ki}\p_i(\a_{kj}\p_jv)
\eeq
and by the first identity in \eqref{nneq1},
\begin{align}
\p_3^2 v = \frac{ \p_3\( b\cdot\naba v\)}{(\a^Tb)_3} +\p_3\( \frac{1}{(\a^Tb)_3}\)b\cdot\naba v-\p_3\(\frac{(\a^Tb)_\beta\p_\beta v}{(\a^Tb)_3}\),
\end{align}
then   the third equation in \eqref{MHDve}  yields
\begin{align}\label{goog1}
(\a^Tb)_3  \p_3 b +\eps \frac{  \a_{k3} \a_{k3}}{(\a^Tb)_3 }   \p_3\( b\cdot\naba v\)  = \naba  q+\mathfrak{f},
\end{align}
where
\begin{align}
\mathfrak{f}:=&- (\a^Tb)_\beta\p_\beta b+\rho\partial_tv-\eps\sum_{i+j<6}\a_{ki}\p_i(\a_{kj}\p_jv)-\eps \a_{k3}\p_3\a_{k3} \p_3 v\nonumber
 \\& -\eps\a_{k3} \a_{k3}\( \p_3\( \frac{1}{(\a^Tb)_3}\)b\cdot\naba v-\p_3\(\frac{(\a^Tb)_\beta\p_\beta v}{(\a^Tb)_3}\)\)- \Psi^\epsd .
\end{align}
Hence, one deduces from \eqref{goog1} that, by \eqref{nndef},
\begin{align}\label{dee1}
  \p_3 b\cdot\tau^\beta  +\eps \frac{ \abs{\n}^2}{(J\a^Tb)_3^2}   \p_3\( b\cdot\naba v\)\cdot\tau^\beta   = G_{\tau^\beta } ,
\end{align}
where
\beq
G_{\tau^\beta }=\frac{1}{(\a^Tb)_3}\(\naba  q+\mathfrak{f}\)\cdot\tau^\beta .
\eeq
Note that $\naba  q \cdot\tau^\beta $ involves only $\p_1q,\p_2q$. Applying $\p^\al$ to \eqref{dee1} then yields \eqref{teqq} with
\begin{align}\label{eeqq2}
 G _{\tau^\beta }^\al= \p^\al  G _{\tau^\beta }-\[\pa^\al, \tau^\beta \]\cdot\p_3 b - \eps \[\p^\al,  \frac{ \abs{\n}^2}{(J\a^Tb)_3^2} \tau^\beta \]\cdot   \p_3\( b\cdot\naba v\).
\end{align}

Now one projects  the third equation in \eqref{MHDve} along $n$. First, the third equation in \eqref{MHDve}  yields
\beq\label{goog2}
  \a_{i3}  \p_3 p+ \a_{i3}\p_3  b_j b_j -b_k \a_{k3}\p_3 b_i-\varepsilon\a_{j3}\p_3 \p_j^\a v_i = \mathfrak{g}_i,
\eeq
where
\beq  \mathfrak{g}_i:=  \varepsilon\a_{j\beta}\p_\beta \p_j^\a v_i +b_k \a_{k\beta}\p_\beta b_i-\a_{i\beta}  \p_\beta p- \a_{i\beta}\p_\beta  b_j b_j -\rho\partial_tv_i + \Psi^\epsd_i .
\eeq
Note that
\begin{align}\label{lb15}
 \a_{i3}   \a_{i3}\p_3  b_j b_j  - \a_{i3}  b_k \a_{k3}\p_3  b_i
 = \a_{i3}   \a_{i3} \( \p_3  b\cdot b  -  \p_3b\cdot n b\cdot n\)= \a_{i3}   \a_{i3}\p_3  b\cdot\tau^\beta   b\cdot \tau^\beta
\end{align}
and
\begin{align}\label{vss1}
-\varepsilon \a_{i3}\a_{j3}\p_3 \p_j^\a v_i
&=-\varepsilon\a_{j3} \p_i^\a  \p_j^\a  v_i  +\varepsilon\a_{j3} \a_{i\beta}\p_\beta \p_j^\a  v_i \nonumber
\\&=-\varepsilon\a_{j3} \a_{j3}\p_3 \diva v -\varepsilon\a_{j\beta} \a_{j\beta}\p_\beta \diva v+\varepsilon\a_{j3} \a_{i\beta}\p_\beta \p_j^\a  v_i .
\end{align}
Hence, one deduces from \eqref{goog2} that, by \eqref{lb15} and \eqref{vss1},
\beq\label{dee2}
   \p_3 p+ \p_3  b\cdot\tau^\beta   b\cdot \tau^\beta   -\varepsilon \p_3 \diva v ={G}_n,
\eeq
where
\beq
{G}_n= \frac{1}{ \a_{i3}   \a_{i3}} \(\mathfrak{g}_j\a_{j3}+\varepsilon\a_{j\beta} \a_{j\beta}\p_\beta \diva v-\varepsilon\a_{j3} \a_{i\beta}\p_\beta \p_j^\a  v_i\).
\eeq
Applying $\p^\al$ to \eqref{dee2} then yields \eqref{teqq2} with
\beq\label{eeqq3}
 G _n^\al= \pa^\al G _n- \[\pa^\al, b\cdot \tau^\beta \tau^\beta \]\cdot \p_3  b.
\eeq

The estimate \eqref{teqqes2}  follows directly from these expressions of  $ G _{\tau^\beta }^\al $ and $G _n^\al$ .
\end{proof}

 \begin{prop}\label{total2}
For $t\in [0,T_0^\epsd]$ with $T_0^\epsd\le 1$, it holds that
\begin{align} \label{total2es}
  \int_0^t \fd{m}^{\varepsilon }
   \le    \me(t)\( \Mm+\eps\Mmd  +  t\fg{m}^\eps(t)  + \eps\fg{m}^\eps(t)     \) .
\end{align}
\end{prop}
\begin{proof}
 Fix first $k=0,\dots, m-1$ and then $\ell=0,\dots, m-k-1$. Let $\al\in \mathbb{N}^{1+3}$ be with $\abs{\al}\le m-1$ such that $\al_0=k$ and $\al_3\le  \ell$.
 First, multiplying both sides of \eqref{teqq} by $(J\a^Tb)_3|\n|^{-1}$ and taking the square to integrate in $\tilde\Omega$, some region near $\Sigma\cup\Sigma_\pm$, one has
\begin{align}
 &\i_\tau ^\al+\sum_{\beta=1,2}\int_{\tilde \Omega}  \(\frac{ (J\a^Tb)_3^2}{|\n|^2}\abs{ \p_3 \p^\al b \cdot \tau^\beta  }^2 +\eps^2 \frac{|\n|^2 }{ (J\a^Tb)_3^2} \abs{   \p_3 \p^\al\( b\cdot\naba v\)\cdot\tau^\beta   }^2\)
 \nonumber\\&\quad = \sum_{\beta=1,2}\int_{\tilde \Omega} \frac{ (J\a^Tb)_3^2}{|\n|^2}\ \abs{G _{\tau^\beta }^\al }^2,
\end{align}
where
\beq
 \i_{\tau   }^\al =2\eps \int_{\tilde \Omega} \p_3 \p^\al b \cdot \tau^\beta      \p_3 \p^\al\( b\cdot\naba v\)\cdot\tau^\beta   .
\eeq
By \eqref{tbb}, one has
\begin{align}\label{ee21}
 \i_{\tau }^\al &=2\eps  \int_{\tilde \Omega} \p_3 \p^\al b \cdot \tau^\beta      \p_3 \p^\al\( \dt b-  \frac{b}{\gamma p}\dt p\)\cdot\tau^\beta   \nonumber
  \\
 &= \eps\dtt \sum_{\beta=1,2}\int_{\tilde \Omega} \abs{ \p_3 \p^\al b \cdot \tau^\beta   }^2   -2\eps  \int_{\tilde \Omega} \p_3 \p^\al b \cdot \tau^\beta      \p_3 \p^\al   b \cdot \dt\tau^\beta   \nonumber
  \\
 &\quad {-2\eps  \int_{\tilde \Omega} \p_3 \p^\al b \cdot \tau^\beta      b  \cdot\tau^\beta   \frac{1}{\gamma p}\p_3 \p^\al  \dt p  }-2\eps \int_{\tilde \Omega} \p_3 \p^\al b \cdot \tau^\beta      \[\p_3 \p^\al , \frac{b}{\gamma p}\] \dt p \cdot\tau^\beta
\nonumber
  \\
 & \ge \eps\dtt \sum_{\beta=1,2}\int_{\tilde \Omega} \abs{ \p_3 \p^\al b \cdot \tau^\beta   }^2   \underline{-2\eps  \int_{\tilde \Omega} \p_3 \p^\al b \cdot \tau^\beta      b  \cdot\tau^\beta   \frac{1}{\gamma p}\p_3 \p^\al  \dt p  }- \eps  \me(\fe{m}^\eps+\fd{m}^\eps) .
\end{align}
 Next,  take the square of \eqref{teqq2} and integrate in ${\tilde \Omega}$ to have
\begin{align}
 \i_{n}^\al+\int_{\tilde \Omega}   \abs{ \p_3\p^\al p+ \p_3 \p^\al b\cdot\tau^\beta   b\cdot \tau^\beta }^2 +\eps^2  \abs{   \p_3 \p^\al \diva v}^2 = \int_{\tilde \Omega}  \abs{G_n^\al }^2,
\end{align}
where
\beq
 \i_n^\al   =-2\eps\int_{\tilde \Omega}   (\p_3\p^\al p+ \p_3 \p^\al b\cdot\tau^\beta   b\cdot \tau^\beta ) \p_3 \p^\al \diva v  .
\eeq
By the second equation in \eqref{MHDve}, one has
 \beq\label{tpp2}
\p^\al   \diva v = -\frac{1}{\gamma p}\dt \p^\al  p-\[\p^\al ,\frac{1}{\gamma p}\]\dt p.
\eeq
Hence,
\begin{align}\label{ee22}
 \i_n^\al & =2\eps\int_{\tilde \Omega}   (\p_3\p^\al p+ \p_3 \p^\al b\cdot\tau^\beta   b\cdot \tau^\beta )\frac{1}{\gamma p} \p_3\dt \p^\al p\nonumber
 \\& \quad+2\eps\int_{\tilde \Omega}   (\p_3\p^\al p+ \p_3 \p^\al b\cdot\tau^\beta   b\cdot \tau^\beta )\(\p_3\(\frac{1}{\gamma p}\)\dt \p^\al p+ \p_3\( \[\p^\al ,\frac{1}{\gamma p}\]\dt p\)\)\nonumber
 \\& = \eps \dtt\int_{\tilde \Omega}  \frac{1}{\gamma p}\abs{ \p_3\p^\al p}^2 -\eps  \int_{\tilde \Omega}   \dt(\frac{1}{\gamma p})\abs{ \p_3\p^\al p}^2     {+2\eps \int_{\tilde \Omega}    \p_3 \p^\al b\cdot\tau^\beta   b\cdot \tau^\beta  \frac{1}{\gamma p} \p_3\dt \p^\al p}
 \nonumber
 \\&\quad+2\eps\int_{\tilde \Omega}   (\p_3\p^\al p+ \p_3 \p^\al b\cdot\tau^\beta   b\cdot \tau^\beta ) \(\p_3\(\frac{1}{\gamma p}\)\dt \p^\al p+ \p_3\( \[\p^\al ,\frac{1}{\gamma p}\]\dt p\)\)\nonumber
 \\& \ge \eps \dtt\int_{\tilde \Omega}  \frac{1}{\gamma p}\abs{ \p_3\p^\al p}^2     \underline{+2\eps  \int_{\tilde \Omega} \p_3 \p^\al b \cdot \tau^\beta      b  \cdot\tau^\beta   \frac{1}{\gamma p}\p_3 \p^\al  \dt p  }-\eps  \me(\fe{m}^\eps+\fd{m}^\eps) .
\end{align}
One can see that it is  crucial that there is an  exact cancelation for the two underlined terms in \eqref{ee21} and \eqref{ee22}, which themselves seem out of control. Hence,  combining \eqref{ee21} and \eqref{ee22} and by \eqref{teqqes2}, one gets
\begin{align}\label{hil1}
&\eps \dtt \int_{\tilde \Omega}\Big( \sum_{\beta=1,2}\abs{ \p_3 \p^\al b \cdot \tau^\beta   }^2 +   \frac{1}{\gamma p}\abs{ \p_3\p^\al p}^2\Big)
   \nonumber
\\&\quad+\int_{\tilde \Omega} \(\frac{ (J\a^Tb)_3^2}{|\n|^2}\sum_{\beta=1,2}\abs{ \p_3 \p^\al b \cdot \tau^\beta  }^2 +   \abs{ \p_3\p^\al p+ \p_3 \p^\al b\cdot\tau^\beta   b\cdot \tau^\beta }^2\)
 \nonumber
\\&\quad+\eps^2 \int_{\tilde \Omega}\(\frac{|\n|^2 }{ (J\a^Tb)_3^2} \sum_{\beta=1,2}\abs{   \p_3 \p^\al\( b\cdot\naba v\)\cdot\tau^\beta   }^2
 +   \abs{   \p_3 \p^\al \diva v}^2\)
 \nonumber
\\&\quad \le  \me\(\sum_{\beta=1,2}\ns{G_{\tau^\beta }^\al }_0+\ns{G_n^\al }_0+ \eps  (\fe{m}^\eps+\fd{m}^\eps)\)
\\&\quad
\le \me\Big( \ns{\p^\al\Psi^\epsd}_0+\mathfrak{F}_{m}^{\varepsilon}+\ns{(\p^\al p,\p^\al b)}_{0,1}+\ns{\p^\al \dt v}_{0}+\eps^2 \ns{\p^\al v}_{1,1}+ \eps  (\fe{m}^\eps+\fd{m}^\eps)\Big) .\nonumber
\end{align}
Integrating \eqref{hil1} in time and using \eqref{dee},
\beq
\p_3v\cdot\tau^\beta =\frac{1}{(\a^Tb)_3}\((b\cdot\naba)v\cdot\tau^\beta -(\a^Tb)_\gamma\p_\gamma v\cdot\tau^\beta \)
\eeq
and
\beq
\p_3v\cdot \n=J \diva v-J\a_{i\gamma}\p_\gamma v_i,
\eeq
one can   deduce that, by \eqref{psies} and \eqref{etaest},
\begin{align}\label{haha111}
& \int_0^t\(\ns{ (\p_3 \p^\al b ,\p_3\p^\al p)}_{L^2(\tilde\Omega)} +\eps^2 \ns{\p^\al v}_{H^2(\tilde\Omega)}\)
 \\&\quad \le  \Mm+\eps\Mmd  +\me(t)\int_0^t\Big( \ns{(\p^\al p,\p^\al b)}_{0,1}+\ns{\p^\al \dt v}_{0}+\eps^2 \ns{\p^\al v}_{1,1}+ \eps  (\fe{m}^\eps+\fd{m}^\eps)\Big) .\nonumber
\end{align}
Consequently, summing \eqref{haha111} over such $\al$ and by \eqref{total1es41}, one obtains
\begin{align} \label{ggod1}
& \int_0^t\(\ns{  (\dt^kp, \dt^kb)}_{\ell+1,m-1-k-\ell} + \varepsilon^2\ns{ \dt^k   v}_{\ell+2,m-1-k-\ell}\) \nonumber
 \\&\quad\le  \Mm+\eps\Mmd  + \me(t)\int_0^t \(\ns{(\dt^kp,\dt^k b)}_{\ell ,m-k-\ell} +\varepsilon^2\ns{  \dt^k  v}_{\ell+1,m-k-\ell}\)\nonumber
  \\&\qquad +\me(t)\int_0^t  \ns{\dt^{k+1} v}_{m-1-k} +\me(t)( t\fg{m}^\eps(t) +  \eps\fg{m}^\eps(t) ) .
\end{align}

Now taking $k=1,\dots,m-1$ in \eqref{ggod1}, by \eqref{tbb} and similarly as in Proposition \ref{tane23},  one deduces that for $\ell=0,\dots, m-k-1$,
\begin{align} \label{ggod2}
& \int_0^t\(\ns{  (\dt^kp,\dt^k b)}_{\ell+1,m-1-k-\ell} + \ns{  \dt^{k-1}  v}_{\ell+2,m-1-k-\ell}  + \varepsilon^2\ns{ \dt^k   v}_{\ell+2,m-1-k-\ell}\) \nonumber
 \\&\quad\le  \Mm+\eps\Mmd  + \me(t)\int_0^t\!\!\!\( \!\ns{(\dt^kp,\dt^k b)}_{\ell ,m-k-\ell} + \ns{  \dt^{k-1}  v}_{\ell+1,m-k-\ell}  +\varepsilon^2\ns{  \dt^k  v}_{\ell+1,m-k-\ell}\!\)\nonumber
  \\&\qquad +\me(t)\int_0^t  \ns{\dt^{k+1} v}_{m-1-k} +\me(t)(t\fg{m}^\eps (t)+  \eps\fg{m}^\eps(t) ) .
\end{align}
A suitable linear combination of   \eqref{ggod2}  for  $\ell=0,\dots, m-k-1$ yields that
\begin{align} \label{ggod3}
& \int_0^t\(\ns{  (\dt^kp,\dt^k b)}_{m-k} + \ns{  \dt^{k-1}  v}_{m-k+1} + \varepsilon^2\ns{ \dt^k   v}_{m-k+1}\) \nonumber
 \\&\quad\le  \Mm+\eps\Mmd  + \me(t)\int_0^t\( \ns{(\dt^kp,\dt^k b)}_{0 ,m-k}+\ns{  \dt^{k-1}  v}_{1,m-k} + \varepsilon^2\ns{  \dt^k  v}_{1,m-k}\)\nonumber
  \\&\qquad +\me(t)\int_0^t  \ns{\dt^{k+1} v}_{m-1-k} +\me(t)( t\fg{m}^\eps(t) +  \eps\fg{m}^\eps(t) )
  \nonumber
 \\&\quad\le  \Mm+\eps\Mmd  +   \me(t)\int_0^t  \ns{\dt^{k+1} v}_{m-1-k} +\me(t)(  t\fg{m}^\eps(t) +  \eps\fg{m}^\eps(t) )  .
\end{align}
Combining   \eqref{ggod3}  for  $k=1,\dots, m-1$  suitably again implies
\begin{align} \label{ggod4}
& \int_0^t\sum_{k=1}^{m-1}\(\ns{ (\dt^kp,\dt^k b)}_{m-k} + \ns{  \dt^{k-1}  v}_{m-k+1}  + \varepsilon^2\ns{ \dt^k   v}_{m-k+1}\) \nonumber
 \\& \quad
 \le  \Mm+\eps\Mmd  +\me(t)    \int_0^t\(\ns{\dt^{m} v}_{0} +\ns{\dt^{m-1} v}_{1}     \)+ \me(t)(  t\fg{m}^\eps(t) +  \eps\fg{m}^\eps(t) )
 \nonumber
 \\& \quad
 \le  \Mm+\eps\Mmd  +  \me(t)(  t\fg{m}^\eps(t) +  \eps\fg{m}^\eps (t)) .
\end{align}
On the other hand, taking $k=0$ in \eqref{ggod1} yields  that for $\ell=0,\dots, m-1$,
\begin{align} \label{ggod5}
& \int_0^t\(\ns{ ( p,b)}_{\ell+1,m-1-\ell}  +  \varepsilon^2\ns{   v}_{\ell+2,m-1-\ell}\) \nonumber
 \\&\quad\le  \Mm+\eps\Mmd  + \me(t)\int_0^t \( \ns{ ( p,b)}_{\ell ,m-\ell}  + \varepsilon^2\ns{  v}_{\ell+1,m-\ell}\) \nonumber
  \\&\qquad +\me(t)\int_0^t    \ns{\dt v}_{m-1}+ \me(t)(  t\fg{m}^\eps (t)+  \eps\fg{m}^\eps(t) )  .
\end{align}
Thus, a suitable linear combination of   \eqref{ggod5}  for  $\ell=0,\dots, m-1$ yields
\begin{align} \label{ggod6}
&  \int_0^t\(\ns{( p,b)}_{m}  + \varepsilon^2\ns{   v}_{m+1}\) \nonumber
 \\&\quad\le  \Mm+\eps\Mmd  + \me(t)\int_0^t\!\!\! \(\!\ns{ ( p,b)}_{0 ,m}   + \varepsilon^2\ns{  v}_{1,m}+\ns{\dt v}_{m-1} \! \)      +  \me(t)(  t\fg{m}^\eps(t) +  \eps\fg{m}^\eps(t))
   \nonumber
 \\& \quad
 \le  \Mm+\eps\Mmd  + \me(t)\int_0^t    \ns{\dt v}_{m-1}+ \me(t)(  t\fg{m}^\eps(t) +  \eps\fg{m}^\eps(t) )  .
\end{align}

Consequently, it follows from \eqref{ggod4} and \eqref{ggod6} that
\begin{align}
& \int_0^t\( \sum_{k=0}^{m-2}\ns{  \dt^{k}  v}_{m-k}+\sum_{k=0}^{m-1}\(\ns{  (\dt^kp, \dt^k b)}_{m-k} +    \varepsilon^2\ns{ \dt^k   v}_{m-k+1}\) \) \nonumber
\\&\quad\le  \Mm+\eps\Mmd  +   \me(t) \(t\fg{m}^\eps (t) +  \eps\fg{m}^\eps(t) \).
\end{align}
This gives \eqref{total2es}.
\end{proof}

\section{$(\epsd)$-independent local well-posedness of \eqref{MHDve}}\label{lasts2}

Now we can conclude the following $(\epsd)$-independent estimates.
\begin{thm}\label{apriori}
There exist positive constants $T_0$ and  $\delta_0$ with the property that for each $0<\delta<\delta_0$, one can find an $\eps_0=\eps_0(\delta)>0$ so that for $0<\eps\le \eps_0$, it holds that
\begin{align} \label{s80}
 \fg{m}^\eps(T_0)    \le \Mm
\end{align}
and that for $t\in [0,T_0],$
\beq\label{zero_12377}
\rho, p, |J|\ge \frac{c_0}{4}>0
\text{ in }\Omega\text{ and } |b\cdot \n|\ge \frac{c_0}{4}>0\text{ on }\Sigma\cup\Sigma_\pm.
\eeq
\end{thm}
\begin{proof}
Note that by the definitions \eqref{mdef1} and \eqref{mdef12} of $\Mm$ and $\Mmd$, for each $0<\delta<\delta_0$ there exists an $\eps_0=\eps_0(\delta)>0$ so that for $0<\eps\le \eps_0$,
\beq\label{mmmddf}
\eps \Mmd \le \Mm.
\eeq
Then for $t\in [0,T_0^\epsd]$ with $T_0^\epsd\le 1$, by the definition \eqref{fgdef} of $\fg{m}^\eps(t)$, \eqref{total1es41}, \eqref{total2es}, \eqref{mmmddf} and \eqref{lalaes}, one deduces
\begin{align} \label{s84}
 \fg{m}^\eps(t) &\le  \me(t)\( \Mm  +\eps \Mmd +  t^{1/2}\fg{m}^\eps(t) +\eps  \fg{m}^\eps(t)  \)
 \nonumber\\  &\le  (\Mm+\eps\Mmd  +t P(\fg{m}^{\varepsilon}(t)))\( \Mm  +\eps \Mmd +  t^{1/2}\fg{m}^\eps(t) +\eps  \fg{m}^\eps(t)  \)
 \nonumber\\&  \le  \Mm   \(1 +t^{1/2} P( \fg{m}^\eps(t))+ \eps \fg{m}^\eps(t)  \).
\end{align}
Therefore,   for $0<\eps\le \eps_0$ by restricting $\eps_0$ smaller if necessary, it holds that
\begin{align} \label{s85}
 \fg{m}^\eps(t)   \le \Mm \(1 +t^{1/2} P( \fg{m}^\eps(t))  \).
\end{align}
From \eqref{s85}, there is a $T_0$, depending on $\Mm$ but not on $\epsd$, such that
\begin{align} \label{s86}
 \fg{m}^\eps(T_0)   \le \Mm  .
\end{align}
This yields \eqref{s80}.  The estimate \eqref{zero_12377} follows from the fundamental theorem of calculus, \eqref{fadl1} and \eqref{s80},  by restricting  $T_0$   smaller if necessary.
\end{proof}

We now present the
\begin{proof}[Proof of Theorem \ref{the_main2}]
With the estimates \eqref{s80} and \eqref{zero_12377} of Theorem \ref{apriori} in hand, following that for the compressible Navier--Stokes equations (see for instance \cite{DS} for the references), one can show routinely that
\begin{align} \label{s87}
 \sup_{[0,T_0]} \mathbb{E}_{2m}+ \int_0^{T_0} \mathbb{D}_{2m}\le  P_\eps( \fg{m}^\eps(T_0),\mathbb{E}_{2m}(0) )\le P_\epsd(\Mm).
\end{align}
The point here is that although the bound of $\mathbb{E}_{2m}(t)$ depends on $\epsd$,  it does not depend on $t\in [0,T_0]$. Now we indicate the dependence of the solutions to \eqref{MHDve} on $\epsd$ by $  (\eta^\epsd ,  p^\epsd,v^\epsd,b^\epsd   )$. Hence, by a standard continuity argument and the local well-posedness recorded in Theorem \ref{lwp}, the unique solution $  (\eta^\epsd ,  p^\epsd,v^\epsd,b^\epsd   )$ above exists actually on $[0,T_0]$ and satisfies the estimate \eqref{enesti2}.  The proof of Theorem \ref{the_main2} is thus completed.
\end{proof}

\section{Local well-posedness of \eqref{MHDv0}}\label{lasts}

Now we present the
\begin{proof}[Proof of Theorem \ref{the_main}]
By Theorem \ref{the_main2},  one can easily pass to the limit as first $\varepsilon\rightarrow 0$ and then $\delta\rightarrow 0$ in \eqref{MHDve} to find that  the limit $(\eta,p,v, b)$ of   $(\eta^\epsd, p^\epsd,v^\epsd, b^\epsd)$, up to extraction of a subsequence, solves  \eqref{MHDv0} on  $[0, T_0]$, where the constraint that $\diva b=0$ in $\Omega$ is recovered by using $(i)$ in Proposition \ref{prop1} as $\diverge_{\a_0}b_0=0$ in $\Omega$. Moreover, $(\eta,p,v, b)$ satisfies
  \beq \label{s91}
  \sup_{[0,T_0]}  \fe{m}^0 +\int_0^{T_0}\fd{m}^0  \le \Mm.
  \eeq
  To improve the estimates, we revisit those estimates in Sections \ref{lasts5} and \ref{lasts4} with setting $\eps=0$ to obtain
\beq \label{s93}
 \fe{m}^0(t) \le\mee^0(t)\Big( \mm+ t \sup_{[0,t]}\se{m}\Big)
\eeq
and
\beq\label{s95}
 \fd{m}^0 \le  \mee^0 \fe{m}^0,
\eeq
respectively.
Then   combining \eqref{s93} and \eqref{s95} yields
\beq \label{s913}
 \se{m}(t) \le\mee^0\Big( \mm+ t \sup_{[0,t]}\se{m}\Big)\le \Mm+ t P(\sup_{[0,t]}\se{m}).
\eeq
From \eqref{s913}, there is a $\tilde T_0$, depending on $\mm$, such that
\begin{align} \label{s8623}
 \se{m} (t)   \le \Mm  ,\ \forall t\in [0,\tilde T_0].
\end{align}
This yields \eqref{enesti} by resetting $T_0$.

Now for the uniqueness, consider two solutions $(\eta^i,p^i,v^i,b^i),\ i=1,2$ to  \eqref{MHDv0} on $[0,T_0]$  with the same   data $(\eta_0,p_0,v_0,b_0,\rho_0)$  which satisfy   the estimate \eqref{enesti}. Denote $\a^i=\a(\eta^i),\ \n^i=\n(\eta^i)$ and $J^i=J(\eta^i)$, $i=1,2.$ Denote $\bar f=f^1-f^2$ for the difference of each unknown $f$, and define
\beq\label{eedef111}
\seb{} :=\sum_{j=0 }^3\norm{(\dt^j \bar p,\dt^j \bar v,\dt^j \bar b)}_{3-j }^2 +\norm{\bar\eta}_3^2 +\abs{\bar\eta}_3^2 .
\eeq
To estimate the differences, one uses again   the   equations \eqref{MHDv0s} with setting $\eps=0$ for each solution. Define again the corresponding  good unknowns as
\begin{equation}
\label{gun1i}
 \mathcal{V}^i =  Z^3  v^i-   Z^3 \eta^i \cdot\nabai  v^i,\  \mathcal{Q}^i = Z^3  q^i-   Z^3 \eta^i \cdot\nabai  q^i
 \eeq
  and
  \beq
 \mathcal{B}^i =  Z^3 b^i-   Z^3 \eta_0   \cdot\nabla_{\a_0} b^i.
\end{equation}
Hence,
\begin{equation}
\label{gun1ib}
 \bar {\mathcal{V}}  =  Z^3  \bar v -    Z^3 \bar \eta \cdot\nabao   v^1 -  Z^3 \eta^2 \cdot \overline{\naba    v},\   \bar {\mathcal{Q}}= Z^3  \bar q -    Z^3 \bar \eta \cdot\nabao   q^1 -  Z^3 \eta^2 \cdot \overline{\naba    q}
 \eeq
  and
  \beq
 \bar {\mathcal{B}}=  Z^3 \bar b -   Z^3 \eta_0   \cdot\nabla_{\a_0} \bar b .
\end{equation}
Note that
\begin{equation}\label{MHDveermm2b}
\begin{split}
 &\jump{ \bar {\mathcal{Q}} }=- (J^1)^{-1} Z^3\bar \eta\cdot\n^1\jump{ \p_3 q^1}-  Z^3 \eta^2 \cdot \overline{J^{-1}\n\jump{\p_3q}},
 \\ &\jump{  \bar {\mathcal{B}} }=-  J_0 ^{-1} Z^3 \eta_0 \cdot\n_0\jump{ \p_3\bar b}, \quad \jump{ \bar {\mathcal{V}}} =0 \text{ on }\Sigma,\quad
 \bar {\mathcal{V}} =0     \text{ on }\Sigma_\pm.
 \end{split}
\end{equation}
One can proceed  as in the proof of Proposition \ref{tane}  to get, by \eqref{enesti},
 \begin{align}\label{ttt1b}
& \dfrac{1}{2}\dfrac{d}{dt}\int_{\Omega} J^1(\frac{1}{\gamma p^1}  \abs{ \bar {\mathcal{Q}}-b^1\cdot  \bar {\mathcal{B}}}^2 +\rho^1 \abs{ \bar {\mathcal{V}} }^2 + \abs{ \bar {\mathcal{B}} }^2 )
 \nonumber
\\&\quad    +\int_{\Omega}J^1 \divao (  \bar {\mathcal{Q}}    \bar {\mathcal{V}})- \int_\Omega J^1b^1\cdot\nabao ( \bar {\mathcal{B}}\cdot  \bar {\mathcal{V}})
 \le \mm  \seb{}.
\end{align}
By \eqref{MHDveermm2b} and since $\dt\bar\eta=\bar v$, one obtains
\begin{align}\label{gg2b}
& -\int_{\Omega}J^1 \divao (  \bar {\mathcal{Q}}    \bar {\mathcal{V}}) =\int_{\Sigma }\jump{ \bar {\mathcal{Q}} }   \bar {\mathcal{V}} \cdot\n^1  \nonumber
 \\&\quad=  -\int_{\Sigma }(     (J^1)^{-1} Z^3 \bar\eta\cdot\n^1\!\jump{ \p_3 q^1}+  Z^3 \eta^2 \cdot \overline{J^{-1}\n\jump{\p_3q}}) ( Z^3  \bar v -    Z^3 \bar \eta \cdot\nabao   v^1 -  Z^3 \eta^2 \cdot \overline{\naba    v})\cdot\n^1\nonumber
  \\&\quad\le  -\int_{\Sigma }(     (J^1)^{-1} Z^3\bar \eta\cdot\n^1\jump{ \p_3 q^1}+  Z^3 \eta^2 \cdot \overline{J^{-1}\n\jump{\p_3q}})  Z^3 \dt \bar \eta \cdot\n^1+\mm  \seb{}
  \nonumber
  \\&\quad\le  -\dtt\int_{\Sigma }\Big( \hal    (J^1)^{-1}\jump{ \p_3 q^1}\abs{ Z^3\bar \eta\cdot\n^1}^2+  Z^3 \eta^2 \cdot \overline{J^{-1}\n\jump{\p_3q}}   Z^3   \bar \eta \cdot\n^1\Big)+\mm  \seb{}
\end{align}
 and
\begin{align}\label{troublesomeb}
& \int_\Omega J^1b^1\cdot\nabao ( \bar {\mathcal{B}}\cdot  \bar {\mathcal{V}})
 =- \int_\Sigma b^1\cdot \n^1   \jump{\bar {\mathcal{B}}} \cdot \bar {\mathcal{V}}
  \nonumber
\\ &\quad=- \int_\Sigma b_0 \cdot \n_0    \jump{\bar {\mathcal{B}}} \cdot ( Z^3  \bar v -    Z^3 \bar \eta \cdot\nabao   v^1 -  Z^3 \eta^2 \cdot \overline{\naba    v})
 \nonumber
\\ &\quad\le  \int_\Sigma    b_0 \cdot \n_0  J_0 ^{-1} Z^3 \eta_0 \cdot\n_0\jump{ \p_3\bar b}\cdot   Z^3  \dt\bar \eta +\mm  \seb{} \nonumber
\\ &\quad\le \dtt \int_\Sigma    b_0 \cdot \n_0  J_0 ^{-1} Z^3 \eta_0 \cdot\n_0\jump{ \p_3\bar b}\cdot   Z^3   \bar \eta +\mm  \seb{} .
 \end{align}
By plugging \eqref{gg2b} and \eqref{troublesomeb} into \eqref{ttt1b} and then integrating in time,  using the fundamental theorem of calculous and the fact that the two solutions take the same initial data, one deduces that
\begin{align} \label{es11db}
  \ns{(\bar {\mathcal{Q}},\bar {\mathcal{V}},\bar {\mathcal{B}})(t) }_0  \le \mm \(  \as{Z^3\bar\eta(t) }_0 +\int_0^t \seb{}  \)  ,
\end{align}
which yields
\begin{align} \label{es11db2}
  \ns{(Z^3\bar p,Z^3\bar v,Z^3\bar b)(t)}_0  \le \mm \(  \as{Z^3\bar\eta(t)}_0 +\int_0^t \seb{}  \) .
\end{align}
Recalling Cauchy's integral   \eqref{re00}, one has
\begin{equation}\label{re00ewe}
\bar b= \rho^1\rho_0^{-1}\a_0^Tb_0 \cdot \nabla \bar \eta+\bar\rho \rho_0^{-1}\a_0^Tb_0 \cdot \nabla   \eta^2.
\end{equation}
Then as in the proof of Proposition \ref{tane}, one can deduce from \eqref{es11db2} that
\begin{align} \label{es11db2b2}
\ns{(\bar p,\bar v,\bar b)(t)}_{0,3}+\as{ \bar\eta (t)}_3  \le \mm  \int_0^t \seb{}   .
\end{align}
On the other hand, one can proceed as in the proof of Propositions \ref{tres1}, \ref{tanest1}, \ref{tane23}, \ref{total2} to estimate the other terms in $\seb{}$, which together with \eqref{es11db2b2} allows one to conclude that
\begin{align} \label{es11db2b2bb}
\seb{}(t) \le \mm \int_0^t \seb{}   .
\end{align}
This and the  Gronwall lemma imply $\seb{}(t)=0$ since $\seb{}(0)=0$, and hence the uniqueness follows. The proof of Theorem \ref{the_main2} is thus completed.
\end{proof}

\appendix

\section{Smoothing initial data}\label{datasec}

We now elaborate the smoothing process that regularizes the data $(\eta_0,p_0,v_0,b_0,\rho_0)$ of \eqref{MHDv0} given in Theorem \ref{the_main} to produce the smooth data $(\eta_0^\delta, p_0^\delta,v_0^\delta, b_0^\delta,\rho_0^\delta) $,  with
the smoothing parameter $\delta>0$.

Let $\Lambda_\delta,\phi_\delta$ be the standard mollifiers in $\mathbb{R}^2$ and $\mathbb{R}^3$, respectively, and $\mathcal{E}_{\Omega_\pm}$ be the Sobolev extension operator.  Let $ \tilde \rho_0^\delta :=\phi_\delta\ast \mathcal{E}_{\Omega_\pm}\rho_{0,\pm} $ in $\Omega_\pm$. Define $\tilde\eta_0^\delta $
as the solution to (see \cite{LM})
\begin{equation}\label{MHDvek0p1}
\begin{cases}
- \Delta^2  \tilde\eta_0^\delta    =- \Delta^2(\phi_\delta\ast \mathcal{E}_{\Omega_\pm}\eta_{0,\pm})  &\text{in } \Omega_\pm\\
\tilde\eta_0^\delta  = \Lambda_\delta\ast\eta_0,\quad  \p_3\tilde\eta_0^\delta  = \Lambda_\delta\ast\p_3\eta_0  &\text{on }\Sigma \cup\Sigma_\pm ,
 \end{cases}
\end{equation}
$\tilde p_0^\delta:=\mathcal{S}_{\Omega_\pm}p_{0,\pm} $ in $\Omega_\pm$ as the solution to
\begin{equation}\label{MHDvek0p2}
\begin{cases}
- \Delta   \tilde p_0^\delta    =- \Delta (\phi_\delta\ast \mathcal{E}_{\Omega_\pm}p_{0,\pm})  &\text{in } \Omega_\pm\\
\tilde p_0^\delta  = \Lambda_\delta\ast p_0  &\text{on }\Sigma \cup\Sigma_\pm ,
 \end{cases}
\end{equation}
$\tilde v_0^\delta:=\mathcal{S}_{\Omega_\pm}v_{0,\pm} $ in $\Omega_\pm$ and $\tilde b_0^\delta:=\mathcal{S}_{\Omega_\pm}b_{0,\pm} $ in $\Omega_\pm$. Then $(\tilde\eta_0^\delta, \tilde p_0^\delta,\tilde v_0^\delta, \tilde b_0^\delta,\tilde  \rho_0^\delta)\in C^\infty(\Omega_\pm) $, and by \eqref{iidd100} and \eqref{compatibility} with $j=0$, one has
\begin{equation}\label{ggdsa}
\jump{\tilde p_0^\delta} =0\text{ and }\jump{\tilde v_0^\delta} =\jump{\tilde b_0^\delta} =\jump{\tilde\eta_0^\delta} =\jump{\p_3\tilde\eta_0^\delta} =0\text{ on } \Sigma,\ \tilde v_0^\delta=0 \text{ and }\tilde\eta_{0,3}^\delta=\pm 1\text{ on }\Sigma_\pm.
\end{equation}
By the standard elliptic theory, it is straightforward to verify that
 \beq\label{convv}
\tilde\eta_0^\delta\rightarrow \eta_0 \text{ in }H^m(\Omega_\pm)\cap H^m(\Sigma)\text{ and } (  \tilde p_0^\delta,\tilde v_0^\delta, \tilde b_0^\delta,\tilde  \rho_0^\delta)\rightarrow (p_0,v_0,b_0,\rho_0) \text{ in }H^m(\Omega_\pm) \text{ as }\delta\rightarrow0.
 \eeq
Moreover, by \eqref{iidd1} and \eqref{convv}, one obtains that for $0<\delta<\delta_0$ with some $\delta_0>0$,
\beq\label{aa001}
\tilde\rho_0^\delta,\tilde p_0^\delta, |\tilde J_0^\delta|\ge \frac{c_0}{2}>0
\text{ in }\Omega\text{ and } |\tilde b_0^\delta\cdot \tilde\n_0^\delta|\ge \frac{c_0}{2}>0\text{ on }\Sigma\cup\Sigma_\pm,
\eeq
where $\tilde J_{0}^\delta=J(\tilde\eta_0^\delta)$ and $\tilde\n_{0}^\delta=\n(\tilde\eta_0^\delta)$.
Note that the smoothed data $(\tilde\eta_0^\delta, \tilde p_0^\delta,\tilde v_0^\delta, \tilde b_0^\delta,\tilde  \rho_0^\delta)$ satisfies the zero-th order compatibility conditions for \eqref{MHDv0} (cf. \eqref{compatibility} for $j=0$), but it may not satisfy any higher order compatibility conditions (cf. \eqref{compatibility} for $j\ge 1$). To adjust this, the idea is to add correctors to $(  \tilde p_0^\delta,\tilde v_0^\delta, \tilde b_0^\delta )$ so that the higher order compatibility conditions  hold. We may assume $\Omega=\mathbb{T}^2\times \mathbb{R}$ for simplicity; in such case, one only needs to add correctors to $(  \tilde p_{0,+}^\delta,\tilde v_{0,+}^\delta, \tilde b_{0,+}^\delta )$ and keep $(  \tilde p_{0,-}^\delta,\tilde v_{0,-}^\delta, \tilde b_{0,-}^\delta )$ unchanged.

To simplify the notations, we denote $U:=(p,v,b)$ and define
\begin{equation}\label{aa11}
\dt^j U^T:= \dt^{j-1}
\begin{pmatrix}
-\gamma p  \diva v   \\
 \rho^{-1}\(b\cdot\naba b-\naba  (p+\hal|b|^2)\)  \\
b\cdot\naba v-b\diva v
\end{pmatrix},\ j\ge 1,
\end{equation}
recursively, where  $\a=\a(\eta)$ and $\dt^{j}\eta=\dt^{j-1}v$ and $\dt^{j}\rho=-\dt^{j-1}(\rho \diva v)$ for $j\ge 1$ have been assumed. One finds that $\dt^j U$ depends on $\nabla\eta,\rho, \tilde\nabla^j U,\p_3^j U$, $j\ge 1$, where $\tilde\nabla^j$ denotes a collection of $\p^\al $ for $\al\in \mathbb{N}^3$ with $|\al|\le j$ and $\al_3<j$. Then we define
\beq\label{aa6}
 {\bf U}_{\nabla\eta,\rho}^j(\tilde\nabla^j U,\p_3^j U):=\dt^j U^T ,\ j\ge 0.
\eeq
Set
\begin{equation}\label{aa00123}
 {\bf V}_{\nabla\eta,\rho}^j(\tilde\nabla^j U,\p_3^j U)
:=
\begin{pmatrix}
  {\bf U}^j_p \\
 {\bf U}^j_v\cdot\tau^1 \\
 {\bf U}^j_v\cdot\tau^2 \\
  {\bf U}^j_v\cdot n \\
   {\bf U}^j_b\cdot\tau^1 \\
    {\bf U}^j_b\cdot\tau^2 \\
\end{pmatrix},\ j\ge 0
\text{ and }
 {\bf W}_{\nabla\eta}^j(\p_3^j U)
:=
\begin{pmatrix}
\p_3^j p \\
\p_3^jv\cdot\tau^1 \\
\p_3^jv\cdot\tau^2 \\
\p_3^jv\cdot n \\
\p_3^jb\cdot\tau^1 \\
\p_3^jb\cdot\tau^2 \\
\end{pmatrix}
,\ j\ge 1,
\end{equation}
where $(\tau^1,\tau^2,n)=(\tau^1,\tau^2,n)(\eta)$ and $  {\bf U}^j_p$ denotes the $p$-component (the first component) of $  {\bf U}^j$, etc. By \eqref{aa11}--\eqref{aa00123}, it holds that
\beq\label{vwf}
{\bf V}_{\nabla\eta,\rho}^j(\tilde\nabla^j U,\p_3^j U)=(\mathbb{E}_{\nabla\eta,\rho}(U))^j {\bf W}_{\nabla\eta}^j(\p_3^j U)+  \mathbb{F}_{\nabla\eta,\rho}^j(\tilde\nabla^j U),\ j\ge 1,
\eeq
where
\beq
\mathbb{E}_{\nabla\eta,\rho}(U):={\tiny\begin{pmatrix}
0&0&0&-\gamma p J^{-1}|\n|&0&0 \\
0&0&0&0&\rho^{-1} J^{-1}b\cdot\n&0\\
0&0&0&0&0&\rho^{-1} J^{-1}b\cdot\n \\
-\rho^{-1}J^{-1}|\n|&0&0&0&-\rho^{-1} J^{-1}|\n|b\cdot\tau^1&-\rho^{-1} J^{-1}|\n|b\cdot\tau^2  \\
0&J^{-1}b\cdot\n&0& J^{-1}|\n|b\cdot\tau^1&0&0  \\
0&0&J^{-1}b\cdot\n& J^{-1}|\n|b\cdot\tau^2&0&0  \\
\end{pmatrix}}.
\eeq
Here for simplicity we have not written out the explicit form of $\mathbb{F}_{\nabla\eta,\rho}^j(\tilde\nabla^j U)$ but  only state its dependences, and the key feature here is that it does not depend on $\p_3^j U$ when $j\ge 1$, that is, we separate the dependences of  ${\bf V}_{\nabla\eta,\rho}^j$ on $\p_3^j U$ and $ \tilde\nabla^j U$.

Now we construct the smoothed data that satisfies the higher order compatibility conditions as follows.
First,  define  $U_0^{\delta,(0)}:=( p_0^{\delta,(0)},v_0^{\delta,(0)}, b_0^{\delta,(0)})=(  \tilde p_0^\delta,\tilde v_0^\delta, \tilde b_0^\delta )$ as the zero-th order smooth approximation of $U_0:=(  p_0, v_0,  b_0 )$. Note that by \eqref{ggdsa}, $(\tilde\eta_0^\delta, U_0^{\delta,(0)},\tilde  \rho_0^\delta)$ satisfies the zero-th order compatibility conditions:
\beq
\jump{ {\bf V}_{\nabla\tilde\eta_0^\delta}^0(U_0^{\delta,(0)})}=0\text{ on }\Sigma,
\eeq
and by \eqref{aa001}, $\mathbb{E}_{\nabla\tilde\eta_{0}^\delta,\tilde\rho_{0,+}^\delta }(U_{0}^{\delta,(0)})$ (here, unlike the other terms, we have specified the subscript ``+" in  $\tilde\rho_{0,+}^\delta$ as $\tilde\rho_{0}^\delta$ may not be continuous across the interface $\Sigma$, etc.) is non-singular on $\Sigma$:
\beq
\det \mathbb{E}_{\nabla\tilde\eta_0^\delta,\tilde\rho_{0,+}^\delta }(U_0^{\delta,(0)})=\gamma \tilde p_0^\delta|\tilde\n_{0}^\delta|^2(\tilde\rho_{0,+}^\delta)^{-3}(\tilde J_{0}^\delta)^{-6}|\tilde b_0^\delta\cdot\tilde\n_{0}^\delta|^4>0\text{ on }\Sigma.
\eeq
Next, suppose that $j\ge 0$ and the $j$-th order smooth approximation $U_0^{\delta,(j)}:=( p_0^{\delta,(j)},v_0^{\delta,(j)}, b_0^{\delta,(j)})$ of $ U_0$ has been constructed so that $(\tilde\eta_0^\delta, U_0^{\delta,(j)},\tilde  \rho_0^\delta)$ satisfies the $j$-th order compatibility conditions:
\beq\label{jcompp}
\jump{ {\bf V}_{\nabla\tilde\eta_0^\delta,\tilde\rho_0^\delta}^\ell(\tilde\nabla^\ell U_0^{\delta,(j)}, \p_3^\ell U_0^{\delta,(j)})}=0,\ \ell=0,\dots,j
\eeq
and
\beq
\det \mathbb{E}_{\nabla\tilde\eta_{0}^\delta, \tilde\rho_{0,+}^\delta}(U_{0}^{\delta,(j)})>0\text{ on }\Sigma.
\eeq
Define
\beq\label{aa16}
\tilde \Phi^{\delta,(j+1)}=-\(\mathbb{E}_{\nabla\tilde\eta_{0}^\delta,\tilde\rho_{0,+}^\delta}(U_{0,+}^{\delta,(j)})\)^{-(j+1)} \jump{ {\bf V}_{\nabla\tilde\eta_0^\delta,\tilde\rho_0^\delta}^{j+1}(\tilde\nabla^{j+1} U_0^{\delta,(j)}, \p_3^{j+1} U_0^{\delta,(j)})} \text{ on }\Sigma
\eeq
and
\beq \label{aa1666}
\Phi^{\delta,(j+1)}:=\begin{pmatrix}
\tilde \Phi^{\delta,(j+1)}_1 \\
\tilde \Phi^{\delta,(j+1)}_2\tilde\tau_0^{1,\delta}+\tilde \Phi^{\delta,(j+1)}_3\tilde\tau_0^{2,\delta}+\tilde \Phi^{\delta,(j+1)}_4\tilde n_0^\delta\\
\tilde \Phi^{\delta,(j+1)}_5\tilde \tau_0^{1,\delta}+\tilde \Phi^{\delta,(j+1)}_6\tilde \tau_0^{2,\delta}
\end{pmatrix}\text{ on }\Sigma,
\eeq
where $(\tilde\tau_0^{1,\delta},\tilde\tau_0^{2,\delta},\tilde n_0^\delta)=(\tau^1,\tau^2,n)(\tilde\eta_0^\delta)$.
Then we construct the $(j+1)$-th order smooth approximation $U_0^{\delta,(j+1)}:=( p_0^{\delta,(j+1)},v_0^{\delta,(j+1)}, b_0^{\delta,(j+1)})$ of $ U_0$   such that
\beq\label{aa099jjj}
\p_3^\ell U_{0,+}^{\delta,(j+1)}= \p_3^\ell U_{0,+}^{\delta,(j)},\ \ell=0,\dots,j\text{ and }\p_3^{j+1} U_{0,+}^{\delta,(j+1)}= \p_3^{j+1} U_{0,+}^{\delta,(j)}
+\Phi^{\delta,(j+1)}\text{ on }\Sigma.
\eeq
The existence of such $U_0^{\delta,(j+1)}$ is standard (see \cite{LM}).
It then follows  that
\beq \label{aa099}
\jump{ {\bf V}_{\nabla\tilde\eta_0^\delta,\tilde\rho_0^\delta }^\ell(\tilde\nabla^\ell U_0^{\delta,(j+1)},\p_3^\ell U_0^{\delta,(j+1)})}=\jump{ {\bf V}_{\nabla\tilde\eta_0^\delta,\tilde\rho_0^\delta}^\ell(\tilde\nabla^\ell U_0^{\delta,(j)},\p_3^\ell U_0^{\delta,(j)})}=0\text{ on }\Sigma,\ \ell=0,\dots,j,
\eeq
\beq
\det \mathbb{E}_{\nabla\tilde\eta_0^\delta,\tilde\rho_{0,+}^\delta }(U_{0,+}^{\delta,(j+1)})=\det \mathbb{E}_{\nabla\tilde\eta_0^\delta,\tilde\rho_{0,+}^\delta }(U_{0,+}^{\delta,(j)})>0\text{ on }\Sigma,
\eeq
\beq\label{aa17}
{\bf W}_{\nabla\tilde\eta_0^\delta}^{j+1}(\p_3^{j+1} U^{\delta,(j+1)}_{0,+})={\bf W}_{\nabla\tilde\eta_0^\delta}^{j+1}(\p_3^{j+1} U^{\delta,(j)}_{0,+})+\tilde \Phi^{\delta,(j+1)}\text{ on }\Sigma
\eeq
and
\beq\label{aa1788}
\mathbb{F}_{\nabla\tilde\eta_0^\delta,\tilde\rho_{0,+}^\delta}^{j+1}(\tilde\nabla^{j+1} U^{\delta,(j+1)}_{0,+})=\mathbb{F}_{\nabla\tilde\eta_0^\delta,\tilde\rho_{0,+}^\delta}^{j+1}(\tilde\nabla^{j+1} U^{\delta,(j)}_{0,+})
\text{ on }\Sigma.
\eeq
Then by \eqref{vwf}, \eqref{aa17} and \eqref{aa16}, one has
\begin{align}
&\jump{ {\bf V}_{\nabla\tilde\eta_0^\delta,\tilde\rho_0^\delta }^{j+1}(\tilde\nabla^{j+1} U_0^{\delta,(j+1)},\p_3^{j+1} U_0^{\delta,(j+1)})}
\nonumber
\\&\quad=\jump{ \(\mathbb{E}_{\nabla\tilde\eta_0^\delta,\tilde\rho_0^\delta }(U_{0}^{\delta,(j+1)})\)^{j+1} {\bf W}_{\nabla\tilde\eta_0^\delta}^{j+1}(\p_3^{j+1} U_0^{\delta,(j+1)})}+ \jump{  \mathbb{F}_{\nabla\tilde\eta_0^\delta,\tilde\rho_0^\delta }^{j+1}(\tilde\nabla^{j+1} U_0^{\delta,(j+1)})}\nonumber
\\&\quad=\jump{ \(\mathbb{E}_{\nabla\tilde\eta_0^\delta,\tilde\rho_0^\delta }(U_{0}^{\delta,(j)})\)^{j+1} {\bf W}_{\nabla\tilde\eta_0^\delta}^{j+1}(\p_3^{j+1} U_0^{\delta,(j+1)})}+ \jump{  \mathbb{F}_{\nabla\tilde\eta_0^\delta,\tilde\rho_0^\delta }^{j+1}(\tilde\nabla^{j+1} U_0^{\delta,(j)})}\nonumber
\\&\quad=\jump{ {\bf V}_{\nabla\tilde\eta_0^\delta,\tilde\rho_0^\delta}^{j+1}(\tilde\nabla^{j+1} U_0^{\delta,(j)},\p_3^{j+1} U_0^{\delta,(j)})}+\(\mathbb{E}_{\nabla\tilde\eta_0^\delta,\tilde\rho_{0,+}^\delta }(U_{0,+}^{\delta,(j)})\)^{j+1}\tilde \Phi^{\delta,(j+1)}
=0.
\end{align}
This together with \eqref{aa099} implies that $(\tilde\eta_0^\delta, U_0^{\delta,(j+1)},\tilde\rho_0^\delta) $ satisfies the $(j+1)$-th order compatibility conditions. Consequently, one can construct recursively the smooth  data that satisfies any high order compatibility conditions.

Now  the  desired smoothed data  is constructed as $(\eta_0^\delta, p_0^\delta,v_0^\delta, b_0^\delta,\rho_0^\delta):=(\tilde\eta_0^\delta, U_0^{\delta,(m-1)},\tilde\rho_0^\delta) $. Note that by \eqref{convv}, $\eta_0^\delta\rightarrow \eta_0$ in $H^m(\Omega_\pm)\cap H^m(\Sigma)$ and $\rho_0^\delta\rightarrow \rho_0$ in $H^m(\Omega_\pm)$ as $\delta\rightarrow0$, however, in general,
$( p_0^\delta,v_0^\delta, b_0^\delta)$ does not converge to $( p_0 ,v_0 , b_0 )$. But if $(\eta_0,p_0,v_0,b_0,\rho_0)$ satisfy the $(m-1)$-th order compatibility conditions:
\beq\label{jcomppjj}
\jump{ {\bf V}_{\nabla \eta_0,\rho_0}^\ell(\tilde\nabla^\ell U_0, \p_3^\ell U_0)}=0,\ \ell=0,\dots,m-1,
\eeq
then $(  p_0^\delta,v_0^\delta, b_0^\delta )\rightarrow ( p_0 ,v_0 , b_0 )$ in $H^m(\Omega_\pm)$ as $\delta\rightarrow0$. Indeed, we claim that $U_0^{\delta,(j)}\rightarrow U_0$ in $H^m(\Omega_\pm)$ as $\delta\rightarrow0$ for $j=0,\dots,m-1$. We prove the claim by induction on $j$. First, for $j=0$, no correctors have been added and so it is direct to have that $U_0^{\delta,(0)}\rightarrow U_0$ in $H^m(\Omega_\pm)$ as $\delta\rightarrow0$. Next, suppose that $j\in [0,m-2]$ and that $U_0^{\delta,(\ell)}\rightarrow U_0$ in $H^m(\Omega_\pm)$ for $\ell=0,\dots,j$  as $\delta\rightarrow0$ have been proved. Then by the definition of the correction $\Phi^{\delta,(j+1)}$ (cf. \eqref{aa16} and \eqref{aa1666}), the induction assumption and \eqref{jcomppjj} with $\ell=j+1$, one deduces that $\Phi^{\delta,(j+1)}\rightarrow0$ in $H^{m-j-3/2}(\Sigma)$ as $\delta\rightarrow0$. So by the definition of $U_0^{\delta,(j+1)}$ (cf. \eqref{aa099jjj}), $U_0^{\delta,(j+1)}-U_0^{\delta,(j)}\rightarrow 0$ in $H^m(\Omega_\pm)$ as $\delta\rightarrow0$, which together with the induction assumption again implies that $U_0^{\delta,(j+1)}\rightarrow U_0$ in $H^m(\Omega_\pm)$ as $\delta\rightarrow0$. The claim is thus proved, and the construction of the smoothed data is completed.

\section{Anisotropic trace estimates}
\begin{lem}\label{tra_es}
Assume that $\bar B\in \Rn{3}$ with  $|\bar B_3|\ge \theta>0$ for $0\le |x_3|\le \iota$.
Then it holds that
\beq\label{sbt}
  \abs{ f}_0^2 \le C_{\bar B,\theta,\iota} \(\norm{\bar B  \cdot \nabla f}_0\norm{ f}_0+ \ns{f}_0 \),
\eeq
where $C_{\bar B,\theta,\iota}$ is a positive constant depending on the $C^1$ norm of $\bar B$, $\theta$ and $\iota$.
\end{lem}
\begin{proof}
Without loss of generality, assume that $\bar B_3\ge \theta>0$ for $0\le x_3\le \iota$.
Set $ \tilde B:= \bar B/\bar B_3=(\tilde B_h ,1)$ and define the {\it 2D} trajectory $Y=Y(x_1,x_2;s)$ by
\beq
\begin{cases}
\dis\frac{ d }{ds}Y(x_1,x_2;s)=\tilde B_h (Y(x_1,x_2;s),s)
\\ Y(x_1,x_2;0)=(x_1,x_2).
\end{cases}
\eeq
Then for any $0\le y_3\le \iota$, by the fundamental theorem of calculus,
\begin{align}
f^2(x_1,x_2,0)&=f^2(Y(x_1,x_2;0), 0)
=f^2(Y(x_1,x_2;y_3), y_3)-\int_0^{y_3}\dts \(f^2(Y(x_1,x_2;s), s)\)ds\nonumber
\\& =f^2(Y(x_1,x_2;y_3), y_3)-2\int_0^{y_3}(\tilde   B\cdot \nabla f f)(Y(x_1,x_2;s), s)\,ds.
\end{align}
Integrating over $(x_1,x_2)\in \mathbb{T}^2$ and using the Fubini theorem, one has
\begin{align}\label{Trac1}
\as{f}_0=\int_{\mathbb{T}^2}f^2(Y(x_1,x_2;y_3), y_3) \,dx_1dx_2
-2\int_0^{y_3}\int_{\mathbb{T}^2}(\tilde B\cdot \nabla f f)(Y(x_1,x_2;s), s)\,dx_1dx_2ds.
\end{align}
For any fixed $0\le s\le \iota,$ by a change of variables $(y_1,y_2)=Y(x_1,x_2;s)$, with the Jacobian $J(x_1,x_2;s):= {\rm det} \nabla Y (x_1,x_2;s)$ given by
\beq
J(x_1,x_2;s)= \exp \(-\int_0^{s}(\p_1 \tilde B_1+\p_2 \tilde B_2)(Y(x_1,x_2;\tau),\tau)d\tau\),
\eeq
one has that for any $0\le s\le \iota,$
\beq
C_{\tilde B,\iota}^{-1}\le J(x_1,x_2;s)\le C_{\tilde B,\iota}:=\exp \(\norm{\p_1 \tilde B_1+\p_2\tilde B_2}_{L^\infty}\iota \).
\eeq
Hence, for any $0\le y_3\le \iota$, it holds that
\begin{align}\label{Trac2}
\int_{\mathbb{T}^2}f^2(Y(x_1,x_2,y_3), y_3) \,dx_1dx_2
&=\int_{\mathbb{T}^2}f^2(y_1, y_2, y_3) J^{-1}(y_1,y_2;y_3)\,dy_1dy_2\nonumber
\\&\le C_{\tilde B,\iota}\int_{\mathbb{T}^2}f^2(y_1, y_2, y_3) \,dy_1dy_2
\end{align}
and similarly, by the Cauchy-Schwarz inequality, one has
\begin{align}\label{Trac3}
&-2\int_0^{y_3}\int_{\mathbb{T}^2}(\tilde B\cdot \nabla f f)(Y(x_1,x_2;s), s)dx_1dx_2ds\nonumber\\
&\quad\le 2 \!\(\!\int_0^{y_3}\int_{\mathbb{T}^2}(\tilde B\cdot \nabla f  )^2(Y(x_1,x_2;s), s)dx_1dx_2ds\!\)^{1/2}\!\!\!\(\!\int_0^{y_3}\int_{\mathbb{T}^2} f  ^2(Y(x_1,x_2;s), s)dx_1dx_2ds\!\)^{1/2}
\nonumber\\
&\quad\le 2C_{\tilde B,\iota}\(\int_0^{y_3}\int_{\mathbb{T}^2}(\tilde B\cdot \nabla f  )^2(y_1,y_2, s)dy_1dy_2ds\)^{1/2}\!\(\int_0^{y_3}\int_{\mathbb{T}^2} f  ^2(y_1,y_2, s)dy_1dy_2ds\)^{1/2}
\nonumber\\
&\quad\le 2C_{\tilde B,\iota} \norm{\tilde   B\cdot \nabla f}_0\norm{ f}_0
\end{align}
Combining \eqref{Trac1}, \eqref{Trac2} and \eqref{Trac3} yields
\begin{align}\label{pp1}
\as{f}_0\le C_{\tilde B,\iota}\int_{\mathbb{T}^2}f^2(y_1, y_2, y_3) dy_1dy_2+2C_{\tilde B,\iota} \norm{\tilde   B\cdot \nabla f}_0\norm{ f}_0.
\end{align}
Integrating \eqref{pp1} over $y_3\in (0,\iota)$, one has
\begin{align}
\iota\as{f}_0\le C_{\tilde B,\iota}\int_0^\iota\int_{\mathbb{T}^2}f^2(y_1, y_2, y_3) dy_1dy_2dy_3+2C_{\tilde B,\iota}\iota \norm{\tilde B\cdot \nabla f}_0\norm{ f}_0,
\end{align}
which implies, recalling $ \tilde B= \bar B/\bar B_3$,
\begin{align}
\as{f}_0&\le C_{\tilde B,\iota}\iota^{-1} \ns{f}_0+2C_{\tilde B,\iota} \norm{\bar B_3^{-1}\bar   B\cdot \nabla f}_0\norm{ f}_0\nonumber
\\&\le C_{\tilde B,\iota}\iota^{-1} \ns{f}_0+4C_{\tilde B,\iota}\theta^{-1} \norm{\bar   B\cdot \nabla f}_0\norm{ f}_0.
\end{align}
This gives \eqref{sbt} by redefining the constant $C_{\bar B,\theta,\iota}$.
\end{proof}




\end{document}